\documentclass{article}

\usepackage[preprint]{neurips_2021}

\usepackage[utf8]{inputenc} 
\usepackage[T1]{fontenc}    
\usepackage[colorlinks]{hyperref}       
\usepackage{url}            
\usepackage{booktabs}       
\usepackage{amsfonts}       
\usepackage{nicefrac}       
\usepackage{microtype}      
\usepackage{xcolor}         

\usepackage{amsthm}
\usepackage{amsmath}
\usepackage{amssymb}

\newtheorem{theorem}{Theorem}
\newtheorem{corollary}{Corollary}
\newtheorem{lemma}{Lemma}
\theoremstyle{definition}

\DeclareMathOperator{\vspan}{span}
\DeclareMathOperator*{\argmin}{argmin}

\usepackage{graphicx}

\usepackage{caption}
\usepackage{subcaption}

\setcitestyle{authoryear,open={(},close={)}} 

\title{The Kernel Perspective on Dynamic Mode Decomposition}

\author{
  Efrain Gonzalez \\
  Department of Mathematics and Statistics\\
  University of South Florida\\
  Tampa, Fl 33620 \\
  \texttt{egonzalez@usf.edu} \\
  \And
  Moad Abudia \\
  School of Mechanical and Aerospace Engineering \\
  Oklahoma State University\\
  Stillwater, OK 74078 \\
  \texttt{abudia@okstate.edu} \\
  \And
  Michael Jury\\
  Department of Mathematics\\
  University of Florida\\
  Gainesville, FL 32608\\
  \texttt{mjury@ufl.edu}
  \And
  Rushikesh Kamalapurkar \\
  School of Mechanical and Aerospace Engineering \\
  Oklahoma State University\\
  Stillwater, OK 74078 \\
  \texttt{rushikesh.kamalapurkar@okstate.edu} \\
  \And
  Joel A. Rosenfeld \thanks{The following YouTube playlist discusses the content of this manuscript: \url{https://youtube.com/playlist?list=PLldiDnQu2phuB64ccOYWxeSBZxq0Gg47g}}\\
  Department of Mathematics and Statistics\\
  University of South Florida\\
  Tampa, Fl 33620 \\
  \texttt{rosenfeldj@usf.edu} \\
}

\begin{document}

\maketitle

\begin{abstract}
  This manuscript revisits theoretical assumptions concerning dynamic mode decomposition (DMD) of Koopman operators, including the existence of lattices of eigenfunctions, common eigenfunctions between Koopman operators, and boundedness and compactness of Koopman operators. Counterexamples that illustrate restrictiveness of the assumptions are provided for each of the assumptions. In particular, this manuscript proves that the native reproducing kernel Hilbert space (RKHS) of the Gaussian RBF kernel function only supports bounded Koopman operators if the dynamics are affine. In addition, a new framework for DMD, that requires only densely defined Koopman operators over RKHSs is introduced, and its effectiveness is demonstrated through numerical examples.
\end{abstract}

\section{Introduction}

Dynamic mode decomposition (DMD) has been gaining traction as a model-free method of making short-run predictions for nonlinear dynamical systems using data obtained as snapshots of trajectories. DMD and its variant, extended DMD, have proven effective at extracting underlying governing principles for dynamical systems from data, and reconstruction methods using DMD allow for the modeling of nonlinear dynamical systems as a sum of exponential functions, which is analogous to models obtained for linear systems \citep{kutz2016dynamic}.

DMD is closely connected to the Koopman operator corresponding to discrete time dynamical systems \citep{schmid_2010}. The Koopman Operator $\mathcal{K}_F$ over a Hilbert function space, $H$, is a composition operator corresponding to the discrete time dynamics, $F : \mathbb{R}^n \to \mathbb{R}^n$, acting on functions in the Hilbert space. Specifically, $\mathcal{K}_Fg = g \circ F$. In the context of discrete time dynamical systems, $F$ is the discrete time dynamics that maps the current state of the system to a future state $x_{i+1} = F(x_i)$. Frequently, the function $g \in H$ is referred to as an \emph{observable}.

DMD aims to obtain a finite rank representation of the Koopman operator by studying its action on the full state observable (i.e. the identity function) \citep{schmid_2010}. Koopman operators over reproducing kernel Hilbert spaces (RKHSs) were studied to take advantage of infinite dimensional feature spaces to extract more information from the snapshots of a system in \cite{williams2015data}. This perspective also enacts a dimensionality reduction by formulating the DMD method in a reproducing kernel Hilbert space (RKHS) framework and implicitly using the kernel trick to compute inner products in the high-dimensional space of observables. In \citep{williams2014kernel}, it is shown that kernel-based DMD produces a collection of Koopman modes that agrees with other DMD results in the literature.

The introduction of kernel based techniques for Koopman analysis and DMD yields a new direction for Koopmanism. Traditionally, the theoretical foundations of Koopman operators, and their implementation in the study of dynamical systems, have been rooted in ergodic theory. When working with functions that are $L^2$ with respect to an ergodic measure, the Birkhoff ergodic theorem guarantees that time averages under the operation of the Koopman operator will converge almost everywhere, or with probability 1 \cite{budivsic2012applied}. This convergence gives invariants for the Koopman operator in $L^2$ with respect to the invariant measure corresponding to the dynamics, and the Wiener Wintner theorem can be leveraged to realize the spectrum of the operators.
 
This is somewhat unsatisfying, since for a deterministic continuous time dynamical system, the ultimate result of discretization and Koopman analysis is probabilistic. Moreover, since the point-wise convergence given by the Birkhoff theorem holds almost everywhere, as opposed to everywhere, there is no guarantee that convergence will be realized at a given point. Furthermore, since all numerical methods deal with finitely many data points, convergence is impossible to verify.
 
The new perspective added by kernel methods is that of approximations. Universal RKHSs, such as those corresponding to Gaussian RBFs and Exponential Dot Product kernel functions, have the ability to approximate any continuous function over a compact subset of $\mathbb{R}^n$ to any desired accuracy up to numerical precision. Moreover, when the kernel function is continuous or bounded, convergence in RKHS norm yields point wise everywhere convergence and uniform convergence over compact subsets of $\mathbb{R}^n$. Given a Koopman operator $\mathcal{K}_F$ corresponding to the discrete dynamics, $F$, an $\epsilon > 0$, and a compact subset $\Omega$ that is invariant for the dynamics, it follows that if $\phi$ satisfies
 \begin{equation}\label{eq:approximate-eigenfunction}
    \| K_F \phi - \lambda \phi \|_H \le \epsilon,
 \end{equation}
 then
 
\begin{equation}\label{eq:pointwise-estimate}
    | \phi(T^m(x)) - \lambda^m \phi(x) | \le C \epsilon \frac{1-\lambda^{m+1}}{1-\lambda}
\end{equation}
for $x \in \Omega$, where the accuracy of the resultant models depend on $\epsilon$ and $\lambda$. Significantly, as $\epsilon$ tends to zero, so does difference $ | \phi(T^m(x)) - \lambda^m \phi(x) | $ at every point in $x \in \Omega$. This everywhere convergence stands in contrast to ergodic methods, where convergence results only hold almost everywhere.
 
The study of dynamical systems through the Koopman formalism over RKHSs manifests as a search for functions that are close to being eigenfunctions of the Koopman operator, rather than the actual eigenfunctions. Since only a finite amount of data can be available for the study of an infinite dimensional operator, actual eigenfunctions typically cannot be computed. In fact, there is no requirement that $\mathcal{K}_F$ even has any eigenfunctions. Instead, the search for ``approximate" eigenfunctions is motivated by the universal approximation properties of kernel functions \cite{steinwart2008support}. This is particularly true of DMD, which attempts to construct a finite rank approximation of a Koopman operator from a collection of observed snapshots. Note that obtaining approximate eigenfunctions as in \eqref{eq:approximate-eigenfunction} is not dissimilar to the objective of ergodic methods, where approximation of system invariants and eigenfunctions using time averages is sought.
The existence eigenfunctions depends on the selection of the Hilbert space, as will be shown in Section \ref{sec:misconceptions}, and eigenfunctions may not be present even in the $L^2$ ergodic setting \cite{budivsic2012applied}.
 
The objective of this manuscript is to present the kernel perspective of Koopmanism as a distinct study from the ergodic perspective. In addition, this manuscript addresses several properties often assumed to hold for Koopman operators, and provides counterexamples showing that these properties do not hold in general for Koopman analysis over RKHSs. Finally, we will give a general and simplified algorithm for DMD using Koopman operators that only requires that the operators be densely defined.

\section{Reproducing Kernel Hilbert Spaces}\label{sec:rkhs}
A Reproducing Kernel Hilbert Space (RKHS) $H$ over a set $X$ is a Hilbert space composed of functions .
003from $X$ to $\mathbb{C}$ such that for all $x \in X$ the evaluation functional $E_x f := f(x)$ is bounded. Therefore, for all $x \in X$ there exists 
a function $K_x \in H$ such that $f(x)= \langle f, K_x\rangle$ for all $f \in H$. The function $K_x$ is the reproducing kernel centered at $x$ and the function $K:X \times X \to \mathbb{C}$ defined by $K(x,y)= \langle K_y,K_x \rangle$ is the unique kernel function corresponding to ${H}$ \citep{aronszajn1950theory}. Throughout most of this manuscript, the methods used will be restricted to RKHSs of real valued functions. However, for some specific examples in Section \ref{sec:misconceptions}, it will be more convenient to employ RKHSs of complex valued functions.

Reproducing kernels can be equivalently expressed as realizations of inner products of feature space mappings in $\ell^2(\mathbb{N})$. In particular, given an orthonormal basis for a RKHS, $\{ e_m(\cdot) \}_{n=1}^\infty \subset H$, the kernel function may be expressed as $K(x,y) = \sum_{m=1}^\infty e_m(x) \overline{e_m(y)}$, where $\Psi(x) := (e_1(x),e_2(x),\ldots) \in \ell^2(\mathbb{N})$ is called a feature map. Equivalently, given a feature mapping $\Psi : X \to \ell^2(\mathbb{N})$, there is a RKHS whose kernel function is given as $K(x,y) = \langle \Psi(x),\Psi(y)\rangle_{\ell^2}.$

We will make repeated use of projections onto finite dimensional vector spaces arising from spans of collections of kernels centered at snapshots from a dynamical system. For a collection of centers $\{ x_1, \ldots, x_m \}$, the projection of a function $g \in H$ onto $\alpha = \vspan\{ K_{x_1},\ldots, K_{x_m}\}$ is given as $\argmin_{h \in \alpha} \| h - g\|$, which can be resolved by expressing $h$ as $h = \sum_{i=1}^m w_i K_{x_i}$, expanding $\|\sum_{i=1}^m w_i K_{x_i} -g\|^2$ via inner products, and setting the derivative with respect to $w := (w_1, \ldots, w_m)^T \in \mathbb{R}^m$ to zero, resulting in the weights
\[ 
    \begin{pmatrix} w_1 \\ \vdots\\ w_m \end{pmatrix} = \begin{pmatrix} K(x_1,x_1) & \cdots & K(x_1,x_m)\\ \vdots & \ddots & \vdots\\ K(x_m,x_1) & \cdots & K(x_m,x_m)\end{pmatrix}^{-1} \begin{pmatrix} g(x_1) \\ \vdots \\ g(x_m) \end{pmatrix}
\]
The projection is then defined by $P_{\alpha} g := \sum_{i=1}^m w_i K_{x_i}$ where the weights $ w_i $ are obtained as above.

A RKHS of real valued functions, $H$, over $\Omega \subset \mathbb{R}^n$, is said to be universal if for any compact $V \subset \Omega$, $\epsilon > 0$, and $h \in C(V)$, there is a function $\tilde h \in H$ such that $\| h - \tilde h\|_\infty < \epsilon$, where $C(V)$ denotes the set of continuous functions defined on $V$. Many commonly used kernel functions satisfy this universality property, including the Gaussian RBF kernel functions and the exponential dot product kernel functions \cite{steinwart2008support}.

\section{Koopman Operators over RKHSs}\label{sec:koopmanoperators}
The theory of Koopman operators has been long intertwined with ergodic theory, where ergodic theoretic methods justify almost everywhere convergence claims of time averaging methods to invariants of the Koopman operator. The Birkhoff and Von Neumann ergodic theorems are posed over $L^1(\mathbb{R})$ and $L^p(\mathbb{R})$ for $p>1$ , respectively \cite{walters2000introduction}. However, the invariants for Koopman operators are not always analytic or even smooth. Hence, ergodic theorems do not give guarantees of convergence within most RKHSs, which are frequently composed of real analytic functions. 
Furthermore, even though ergodic theorems guarantee the existence of invariants over $L^2(\mathbb{R})$, the invariant itself is hidden behind a limiting operation via time averages \cite{walters2000introduction}. Hence, there is an expected error in the constructed invariant that stems from finiteness of data.

The objective of DMD methods is to find functions within the function space that nearly achieve eigenfunction behavior. Specifically, for a Koopman operator, $\mathcal{K}_F$, and $\epsilon > 0$, the objective is to find $\hat\varphi \in H$ and $\lambda \in \mathbb{C}$ for which $| \mathcal{K}_F \hat\varphi(x) - \lambda \hat \varphi(x)| < \epsilon$ for all $x$ in a given workspace. When such a function is discovered, the eigenfunction witnesses the snapshots, $x_{i+1} = F(x_i)$, as an exponential function as $\hat \varphi(x_{i+1}) = \lambda^{i+1} \hat \varphi(x_0) + \frac{1-\lambda^{i+1}}{1-\lambda} \cdot \epsilon.$  If $\hat\varphi$ is a proper eigenfunction for $\mathcal{K}_F$, then $\epsilon$ may be taken to be zero, and $\hat\varphi(x_{i+1}) = \lambda^{i+1} \hat\varphi(x_0)$.

Ergodic methods generally yield $| \mathcal{K}_F \hat\varphi(x) - \lambda \hat \varphi(x)| < \epsilon$ only for almost all $x$ within the domain of interest. For a RKHS, the condition $| \mathcal{K}_F \hat\varphi(x) - \lambda \hat \varphi(x)| < \epsilon$ may be relaxed to $\| \mathcal{K}_F \hat\varphi - \lambda \hat \varphi\|_H < \epsilon$, since $| \mathcal{K}_F \hat\varphi(x) - \lambda \hat \varphi(x)| < C \| \mathcal{K}_F \hat\varphi - \lambda \hat \varphi\| < C \epsilon$, where $C > 0$ depends on the kernel function and the point $x$. If the kernel function is continuous and the domain is compact, a finite $C$ may be selected uniformly for that domain. In the special case of the Gaussian RBF kernel function and the domain being $\mathbb{R}^n$, $C$ may be taken to be $1$.

If $\mathcal{K}_F$ is compact and the finite rank approximation of $\mathcal{K}_F$, which we will denote as $\hat{\mathcal{K}}_F$, is within $\epsilon$ of $\mathcal{K}_F$ with respect to the operator norm, and if $\hat\varphi$ is a normalized eigenfunction for $\hat{\mathcal{K}}_F$ with eigenvalue $\lambda$, then $\| \mathcal{K}_F \hat \varphi - \lambda \hat \varphi\|_H \le \| \mathcal{K}_F \hat \varphi - \hat{\mathcal{K}}_F \hat \varphi \|_H \le \| \mathcal{K}_F - \hat{\mathcal{K}}_F\| \le \epsilon.$ Hence, if we obtain a finite rank approximation of $\mathcal{K}_F$ that is within $\epsilon$ with respect to the operator norm, then an eigenfunction of the finite rank approximation will approximate the behavior of an eigenfunction of $\mathcal{K}_F$. This approximation is important in DMD, where the eigenfunctions are utilized to generate an approximation of the full state observable, $g_{id}(x) := x$, one dimension at a time, as outlined in Section \ref{sec:algorithm}.

If an accurate finite rank approximation of the Koopman operator can be obtained in a RKHS, then the approximation of the overall model is accurate point-wise everywhere. In contrast, the approximation is only accurate almost everywhere when considering Koopman operators posed over $L^2(\mathbb{R})$. Point-wise everywhere approximation is a distinctive advantage of kernel based methods that is not clear from an invocation of the kernel trick in machine learning, and requires the operator theoretic considerations introduced in this manuscript.

The need for approximation of the Koopman operator in the operator norm topology naturally leads to the question of when a Koopman operator, or more generally, a composition operator, can be compact. Compactness is a central issue for DMD as every approximation is indeed of finite rank, stemming from the observed data. In 1979, it was established in \cite{singh1979compact} that composition operators over $L^2(\mu)$ cannot be compact when $\mu$ is non-atomic. Indeed, $L^2(\mathbb{R})$ has no compact composition operators.

Most Koopman operators over most frequently considered RKHSs are compact for a very narrow range of dynamics. In fact, most Koopman operators over these spaces are not even bounded, as will be expanded upon in Section \ref{sec:boundednessofkoopman}. Unboundedness is an added complication for the valid implementation of DMD, addressed in Section \ref{sec:algorithm}, using densely defined and potentially unbounded Koopman operators. Alternatively, other classes of compact operators over RKHSs connected to dynamical systems can be leveraged for DMD procedures to give convergence guarantees (see, e.g., \cite{rosenfeld2022dynamic}).

The remainder of this section introduces densely defined Koopman operators over RKHSs, where Lemma \ref{lem:denselydefinedkoopmanproperty} enables the DMD algorithm introduced in Section \ref{sec:algorithm}. Let $H$ be a RKHS over $\mathbb{R}^n$. For a function $F:\mathbb{R}^n \to \mathbb{R}^n$ we define the Koopman Operator (sometimes called a composition operator), $\mathcal{K}_F : \mathcal{D}(\mathcal{K}_F) \to H$, as $\mathcal{K}_F g = g\circ F$ where $\mathcal{D}(\mathcal{K}_F)=\{g \in H : g \circ F \in H\}$. When $\mathcal{D}(\mathcal{K}_F)$ is dense in $H$, $\mathcal{K}_F$ is said to be densely defined. While not all densely defined Koopman operators over RKHSs are bounded, they are all closed operators.

\begin{lemma}\label{lem:denselydefinedkoopmanproperty} 
Let $F : X \to X$ be the symbol for a Koopman operator over a RKHS $H$ over a set $X$. $\mathcal{K}_F : \mathcal{D}(\mathcal{K}_F) \to H$ is a closed operator. 
\end{lemma}

\begin{proof}Suppose that $\{ g_m \}_{m=1}^\infty \subset \mathcal{D}(\mathcal{K}_F)$ such that $g_m \to g \in H$ and $\mathcal{K}_F g_m \to h \in H$. To show that $\mathcal{K}_F$ is closed, we must show that
$g \in \mathcal{D}(\mathcal{K}_F)$ and $\mathcal{K}_F g = h$. This amount to showing that $h = g\circ F$, by the definition of $\mathcal{D}(\mathcal{K}_F)$. Fix $x \in X$, then 
\begin{gather*}h(x) = \langle h, K_x\rangle = \lim_{m \to \infty} \langle \mathcal{K}_F g_m, K_x \rangle= \lim_{m\to\infty} g_m(F(x)) \\= \lim_{m\to\infty} \langle g_m, K_{F(x)} \rangle= \langle g, K_{F(x)}\rangle = g(F(x)).\end{gather*}
As $x$ was an arbitrary point in $X$, $h = g(F(x))$ and the proof is complete.
\end{proof}

If a Koopman operator is densely defined, its adjoints are densely defined and closed. Given the kernel function centered at $x$, $\langle \mathcal{K}_Fg, K_x\rangle=\langle g\circ F,K_x\rangle =g(F(x))=\langle g,K_{F(x)}\rangle$. Thus, $K_{x} \in \mathcal{D}(\mathcal{K}_F^*)$ for all $x \in X$, and $\mathcal{K}_F^*K_x=K_{F(x)}$. Hence, each kernel function is in the domain of the adjoint of a densely defined Koopman operator, and as the span of kernel functions is dense in their RKHS, the adjoint is densely defined.

\section{A Different Landscape for Koopman Operators}\label{sec:misconceptions}

This section examines properties of Koopman operators over RKHSs. The selection of space fundamentally changes the behavior of Koopman operators over that space, where properties such as the lattice of eigenfunctions, common eigenfunctions for different discretizations, and boundedness of the operators may not hold. Each succeeding subsection provides counter examples for each of these properties for specific spaces, and it is proven that the Gaussian RBF's native space only supports bounded Koopman operators if the discrete dynamics are affine.


Recently, kernel methods have been adapted for the study of DMD and Koopman operators, largely through the guise of extended DMD, where kernels are leveraged to simplify computations via the kernel trick. However, the adjustment from the classical study of Koopman operators through ergodic theory to that of reproducing kernel Hilbert spaces leads to significant differences in the Koopman operators and their properties. In most cases, the ergodic theorem cannot be directly applied to recover invariants of Koopman operators, since those invariants are often nonsmooth. This present section exemplifies some of the distinguishing properties of Koopman operators over RKHSs, and in some cases illustrates their limitations.

Much of the classical properties of Koopman operators in a variety of specific contexts, such as $L^2$ spaces of invariant measures and $L^1$ spaces, can be seen in \citep{budivsic2012applied,kawahara2016dynamic,kutz2016dynamic,brunton2019data,brunton2021modern}. Properties of the Koopman operator strongly depend on the selection of underlying vector space, and boundedness, compactness, eigenvalues, etc. change based on this selection. While Koopman operators were introduced by Koopman in 1931 in \citep{koopman1931hamiltonian} and then later picked up by the data science community in the early 2000s (e.g. \citep{mezic2005spectral,kutz2016dynamic}), the study of such operators and their properties continued in earnest throughout the 20th century as composition operators (e.g. \citep{shapiro2012composition}). This is particularly important for RKHSs, where the specification of a bounded or densely defined Koopman operator over a particular space yields strong restrictions on the available dynamics. 

\subsection{Concerning Sampling and Discretizations}
\subsubsection{Forward Complete Dynamics}

In applications, Koopman operators enter the theory of continuous time dynamics through a discretization of the continuous time dynamical system \citep{bittracher2015pseudogenerators,mauroy2016global}. That is, given the dynamical system $\dot x = f(x)$, the system is discretized through the selection of a fixed time-step, $\Delta t > 0$, as $x_{m+1} = x_m + \int_{t_m}^{t_{m} + \Delta t} f(x(t)) dt$, where the right hand side plays the role of the discrete dynamics. However, for such a discretization to exist for arbitrarily large values of $m$, it is necessary that the dynamics be \textit{forward complete}.

For example, consider the one dimensional dynamics, $\dot x = 1 + x^2$. For fixed $0 < \Delta t < \pi/2$, the corresponding discrete time dynamics are given as $x_{m+1} = \tan(\arctan(x_m) + \Delta t)$. Setting $x_m = \tan(\pi/2 - \Delta t)$, it is clear that $x_{m+1}$ is undefined. Consequently, the composition symbol, $F(x) = \tan(\arctan(x) + \Delta t)$ for the hypothetical Koopman operator, is not well defined over $\mathbb{R}^n$ for any selection of $\Delta t$.

The forward completeness assumption restricts the class of continuous dynamics that Koopman based methods may be applied. A DMD method that circumvents this requirement, by utilizing Liouville operators and occupation kernels, may be found in \citep{rosenfeld2022dynamic}.

\subsubsection{Sampling and Data Science}

Ergodic based methods as employed in \cite{budivsic2012applied,mezic2005spectral,kutz2016dynamic,kutz2016dynamic,takeishi2017learning} provide a methodology for obtaining invariants and eigenfunctions for a Koopman operator almost everywhere. That is, by selecting a continuous representative of an equivalence class in an $L^2$ space for the invariant measure, at almost every point within the domain, time averaging against that representative will converge to an invariant of the operator. However, this is a ``probability 1'' result, and the number of points where it may fail can potentially be uncountable. Without any external information concerning the convergence, there is no true guarantee that at particular selected point, the time-averaged approximation will be close to the value of an actual invariant at that point. Such computational issues is precisely where the strength of kernel methods manifests.

To illustrate the kernel method, suppose that $F: \mathbb{R}^n \to \mathbb{R}^n$ is a discretization of a dynamical system with the corresponding Koopman operator, $\mathcal{K}_F : H \to H$, and $H$ is a RKHS over $\mathbb{R}^n$ consisting of continuous functions. Suppose further that $\epsilon > 0$ and $\hat{\mathcal{K}}_F : H \to H$ is an approximation of $\mathcal{K}_F$ such that the norm difference is bounded as $\| \mathcal{K}_F - \hat{\mathcal{K}}_F \| < \epsilon$.

Suppose that $\hat{\varphi} \in H$ and is a normalized eigenfunction of $\hat{\mathcal{K}}_F$ with eigenvalue $\lambda$. The function $\hat{\varphi}$ behaves, pointwise, as an approximate eigenfunction of $\mathcal{K}_F$, since
\begin{gather*}
    | \mathcal{K}_F \hat{\varphi}(x) - \lambda \hat{\varphi}(x) | = | \mathcal{K}_F \hat{\varphi}(x) - \hat{\mathcal{K}}_F(x) \hat{\varphi}(x) |\\
    = | \langle (\mathcal{K}_F - \hat{\mathcal{K}}_F) \hat{\varphi},K(\cdot,x)\rangle_H | \le \|\mathcal{K}_F - \hat{\mathcal{K}}_F\| \|K(\cdot,x)\| \le \epsilon \cdot C\\
\end{gather*}
and $C > 0$ is a constant that depends on the kernel function and a prespecified compact domain. The compact domain may be extended to all of $\mathbb{R}^n$ in some cases, such as when the kernel function is the Gaussian RBF kernel function. Thus, it can be seen that kernel spaces and approximations that are close to the Koopman operator, in operator norm, can provide functions that behave similar to eigenfunctions of the Koopman operator. Moreover, the difference in behavior from a proper eigenfunction is governed pointwise by how close the operator approximation is in the first place.

\subsection{Properties of the Operators}

In this section, 
we consider a single simple example. Consider the dynamical system $\dot x = \begin{pmatrix} x_2 & -x_1 \end{pmatrix}^T$, which corresponds to circular dynamics in the plane. For any fixed $\theta := \Delta t$, the discretization of this system yields the linear discrete dynamics $x_{m+1} = \begin{pmatrix} \cos(\theta) & -\sin(\theta) \\ \sin(\theta) & \cos(\theta) \end{pmatrix} x_{m}$. That is, the discretization corresponding to a fixed time-step results in a fixed rotation of $\mathbb{R}^2$. To simplify the presentation, 
we use $\mathbb C$ as a model for $\mathbb R^2$, where rotation of the complex plane reduces to multiplication by a unimodular constant, $z_{m+1} = e^{i\theta} z_m$. The corresponding discrete time dynamics will be written as $F_\theta(z) := e^{i\theta}z$.

As a function space for definition of the Koopman operator, this section will consider the classical Fock space consisting of entire functions. The Fock space is used extensively in Quantum Mechanics \citep{hall2013quantum} and it is a space where operators have been well studied \citep{zhu2012analysis}. The Fock space is given as
\[ F^2(\mathbb{C}) := \left\{ f(z) = \sum_{m=0}^\infty a_m z^m : \sum_{m=0} |a_m|^2 m! < \infty \right\}. \]
The Fock space is a RKHS, with kernel function $K(z,w) = e^{\bar w z}$. Kernel function for the Fock space over $\mathbb{C}^n$ may be obtained through a product of single variable kernels as $K(z,w) = e^{w^* z} = e^{\bar w_1 z_1} \cdots e^{\bar w_n z_n}.$

Closely related to the Fock space is the exponential dot product kernel, $e^{x^Ty}$, where for a single variable, the exponential dot product kernel's native space may be obtained by restricting the Fock space to the reals, and then taking the real part of the restricted functions. Through a conjugation of the exponential dot product kernel, the Gaussian RBF may be obtained as $K_G(x,y) = e^{-\|x\|^2/2} e^{x^Ty} e^{-\|y\|^2/2} = \exp\left(-\frac{\|x-y\|^2}{2}\right),$ and performing the same operation on the Fock space kernel over $\mathbb{C}^n$ yields $K_G(z,w) = e^{-z^2/2} e^{ w^* z} e^{-\bar w^2/2} = \exp\left( -\frac{(z - \bar w)^2}{2}\right)$, which is the kernel corresponding to the complexified native space for the Gaussian radial basis function over $\mathbb{C}^n$ (cf. \citep{steinwart2008support}). This space may be expressed as
\[ H_{G}^2(\mathbb{C}) = \left\{ g(z) e^{-z^2/2} : g \in F^2(\mathbb{C}^n) \right\}, \]
and the native space corresponding to the Gaussian RBF can be obtained by taking the real parts of functions from $H_G^2$ and restricting to 
$\mathbb R^n$.

\subsubsection{Lattice of Eigenfunctions}

As presented in \citep{budivsic2012applied,klus2015numerical}, the eigenfunctions of Koopman operators over $L^1(\mathbb{R})$ form a lattice. That is if $\varphi_1$ and $\varphi_2$ are two eigenfunctions for the Koopman operator, then so is $\varphi_1 \cdot \varphi_2$. For the lattice to occur more generally, it is necessary for the product of the eigenfunctions to be a member of the underlying vector space. This closure property holds, for example, in the space of continuous functions and other Banach algebras. Hilbert spaces are not generally Banach algebras, and since it is desirable to work over Hilbert spaces for properties such as best approximations, projections, and orthonormal bases (cf. \citep{folland1999real}), it is important to demonstrate that the closure property of eigenfunctions of Koopman operators does not hold in general.

Setting $\theta = \pi$, the discrete dynamics corresponding to rotation by $\pi$ in the complex plane becomes $z_{m+1} = e^{i\pi} z_m = -z_m$. That is the corresponding Koopman operator, $\mathcal{K}_{F_\pi} : F^{2}(\mathbb{C}) \to F^2(\mathbb{C})$, is given as $\mathcal{K}_{F_{\pi}} g(z) := g(-z)$. Hence, every even function is an eigenfunction for this Koopman operator with eigenvalue $1$.

Any function $g \in F^2(\mathbb{C})$ exhibits a strict bound on its growth rate (cf. \citep{zhu2012analysis}). To wit, $|g(z)| = |\langle g, K(\cdot,z) \rangle_{F^2(\mathbb{C})}| \le \|g\|_{F^2(\mathbb{C})}\| K(\cdot,z)\|_{F^2(\mathbb{C})} = \| g\|_{F^2(\mathbb{C})} e^{\frac{|z|^2}{2}}.$ That is, if a function is in the Fock space then the function is of order at most $2$, and if the function is of order $2$ it has type at most $1/2$ (cf. \citep{boas2011entire}). Conversely, if an entire function is of order less than $2$, it is in the Fock space, and if it is of order $2$ and type less than $1/2$, then it is also in the Fock space. While functions of order $2$ and type $1/2$ can be in the Fock space, it does not hold for every such function. For example, $e^{z^2/2}$ is of order $2$ and type $1/2$, but is not in the Fock space.

Thus, $\varphi(z) = e^{z^2/4}$ is an eigenfunction for $\mathcal{K}_{F_{\pi}}$ in the Fock space. However, $\varphi\cdot\varphi = e^{z^2/2}$ is not in the Fock space, and cannot be an eigenfunction for $\mathcal{K}_{F_{\pi}}:F^2(\mathbb{C}) \to F^2(\mathbb{C})$. \textit{Hence, the eigenfunctions for $\mathcal{K}_{F_{\pi}}$ do not form a lattice.}

\subsubsection{Common Eigenfunctions}

The intuition behind the use of Koopman operators in the study of continuous time dynamical systems is that eigenfunctions for the Koopman operators should be ``close'' to that of the Koopman generator for small timesteps. However, semi-groups of Koopman operators do not always share a common collection of eigenfunctions.

For example, set $\theta = \pi/2$, which yields $F_{\pi/2}(z) = i z$. In this case, the polynomial $z^4 + z^8 \in F^2(\mathbb{C})$ is an eigenfunction for the Koopman operator corresponding to $\theta = \pi/2$ with eigenvalue $1$. However, $z^4 + z^8$ is not an eigenfunction for $\mathcal{K}_{F_{\pi/3}}$, as $(e^{i \pi/3} z)^4 + (e^{i \pi/3} z)^8 = e^{i 4\pi/3} z^4 + e^{i 2\pi/3} z^8$, and the constants cannot be factored out of the polynomial as an eigenvalue.

Hence, since each Koopman operator obtained through a fixed time-step may produce a different collection of eigenfunctions, there is no way to distinguish which, if any, should correspond to eigenfunctions of the Koopman generator.

\subsubsection{Boundedness of Koopman Operators}\label{sec:boundednessofkoopman}

Throughout the literature, it is frequently assumed that Koopman operators are bounded.
This assumption manifests as an unrestricted selection of observables in the study of the Koopman operator. When a Koopman operator is a densely defined operator whose domain is the entire Hilbert space, it is also closed. Hence, by the closed graph theorem (cf. \citep[Theorem 5.12]{folland1999real}), 
such an operator must be bounded. Furthermore, the collection of finite rank operators is dense in the collection of bounded operators over a Hilbert space in the strong operator topology (SOT) (cf. \citep[Paragraph 4.6.2]{pedersen2012analysis}). Convergence in SOT was independently studied in the work \citep{korda2018convergence}, where the DMD routine was demonstrated to converge to a bounded Koopman operator in SOT.

As mentioned in \citep{korda2018convergence}, SOT convergence does not in general lead to convergence of the eigenvalues. To preserve spectral convergence, the finite rank approximations produced by DMD algorithms need to converge to Koopman operators in the operator norm topology. The most direct approach, and one that leads to good pointwise estimates of eigenfunctions, is through the use of compact Koopman operators. However, it isn't immediately clear when one can expect a continuous dynamical system to yield a compact Koopman operator through discretization. For example, the Koopman operator corresponding to discretization of the continuous time system $\dot x = 0$ is the identity operator, $I$, for any fixed time step, and $I$ is not compact over any infinite dimensional Hilbert space.

In addition, for any given RKHS, the collection of bounded Koopman operators is very small. It was demonstrated in \citep{carswell2003composition} that a Koopman operator over the Fock space is bounded only when the corresponding discrete dynamics are \textit{affine}. It follows that the same result holds over the exponential dot product kernel's native space.

It may perhaps be less obvious that this result extends to the Gaussian RBF's native space, and a proof of this fact is given below. As far as the authors are aware, this is the first time this result has appeared in the literature, and it demonstrates that even for popular selections of RKHSs, the collection of bounded Koopman operators is small.
\begin{lemma}
If $\mathcal{K}_F$ is a bounded operator over the Gaussian RBF's native space, then $F$ is a real analytic vector valued function over $\mathbb{R}^n$.
\end{lemma}
\begin{proof}
If $\mathcal{K}_F$ is bounded, then $\mathcal{K}_FK_y(x) = K_y(F(x)) = \exp(-\|F(x) - y\|^2)$ is in the RBF's native space for each $y \in \mathbb{R}^n$. Since every function in the RBF's native space is real analytic, so is $K_y(F(x))$, and thus, the logarithm, $-\|F(x) - y\|^2 = -\|F(x)\|^2 +2y^TF(x) - \|y\|^2$ is real analytic. This holds if $y=0$, and hence $\|F(x)\|^2$ is real analytic. Thus, for every $y$, the quantity $y^T F(x)$ is real analytic. That each component of $F(x)$ is real analytic follows from the selection of $y$ as the cardinal basis elements of $\mathbb{R}^n$, and this completes the proof.  
\end{proof}
\begin{lemma}If $F$ is a real analytic vector valued function that yields a bounded Koopman operator, $\mathcal{K}_F$, over a the Gaussian RBF's native space, then its extension to an entire function, $F:\mathbb{C}^n \to \mathbb{C}^n$ induces a bounded operator over $H_G(\mathbb{C}^n)$.\end{lemma}
\begin{proof}
Since an entire function on $\mathbb{C}^n$ is uniquely determined by its restriction to $\mathbb{R}^n$, it follows that the span of the complex valued Gaussian RBFs with centers in $\mathbb{R}^n$ is dense in $H_G$. Moreover, to demonstrate that $\mathcal{K}_{F}$ is bounded, it suffices to show that there is a constant $C$ such that
\begin{equation}\label{eq:postivedefbounded}C^2 K_G(z,w) - K_G(F(z),F(w))\end{equation} is a positive kernel. By the above remark, it suffices to show this for real $x,y \in \mathbb{R}^n$, but then this is equivalent to the statement that $\mathcal{K}_F$ is bounded over the Gaussian RBF's native space over $\mathbb{R}^n$. 
\end{proof}
\begin{theorem}
If $F:\mathbb{C}^n \to \mathbb{C}^n$ is an entire function, and $\mathcal{K}_F$ is bounded on $H_G$, then $F(z) = Az + b$ for a matrix $A \in \mathbb{C}^{n\times n}$ and vector $b \in \mathbb{C}^n$.
\end{theorem}
\begin{proof}
If $\mathcal{K}_F$ is bounded, then it has a bounded adjoint, $\mathcal{K}_F^*$, which acts on the complex Gaussian as $\mathcal{K}_F^* K_G(\cdot,z) = K_G(\cdot,F(z))$. In particular, there is a constant $C > 0$ such that $\| K_G(\cdot,F(z)) \|^2 \le C^2 \| K_G(\cdot,z) \|^2$. Noting the identity $\|K_G(\cdot,z)\|^2 = \exp\left(2 \sum_{j=1}^n (\mathcal{I}z_j)^2\right)$ and taking the logarithm, it follows that
\begin{equation}
    \sum_{j=1}^n (\mathcal{I}F_j(z))^2 \le \log(C^2) + \sum_{j=1}^n (\mathcal I z_j)^2 \le \log(C^2) + \|z\|^2.
\end{equation}
From this inequality, it follows that for each coordinate $j=1,\ldots,n$, the harmonic function $v_j(z) = \mathcal{I}F_j(z)$ has linear growth. That is, there is a constant $\tilde C$ so that $|v_j(z)| \le \tilde C (1 + \| z \|)$ for all $z \in \mathbb{C}^n$. It follows (e.g. from the standard Cauchy estimates) that $v_j(z) = v_j(x+iy)$ must be an affine linear function of $x$ and $y$, and therefore, so must its harmonic conjugate $u_j(z)$, and we conclude that $F(z) = Az + b$.
\end{proof}
\begin{corollary}
    If $F$ is a real entire vector valued function, and $\mathcal{K}_F$ is bounded on the Gaussian RBF's native space over $\mathbb{R}^n$, then $F$ is affine.
\end{corollary}
Hence, for the most commonly used kernel function in machine learning, the collection of bounded (and hence compact) Koopman operators over its native space is restricted to only those Koopman operators corresponding to affine dynamics. Each selection of RKHS and kernel function will yield a correspondingly small collection of bounded Koopman operators. It should be noted that Koopman operators were completely classified for over the classical sampling space, the Paley-Wiener space \citep{chacon2007composition}, as also being those that correspond to affine dynamics, and it is a simple exercise to show that the native space for the polynomial kernel also only admits bounded Koopman operator when the dynamics are affine.

Consequently, in most practical respects Koopman operators over RKHSs should not be assumed to be bounded, and certainly not compact.

\section{Dynamic Mode Decomposition with Koopman Operators over RKHSs\label{sec:algorithm}}

As a product of its genesis in the machine learning community, many DMD procedures appeal to feature space, and this continues to hold in the current implementations of kernel-based extended DMD \citep{williams2014kernel}, which casts the snapshots from a finite dimensional nonlinear system into an infinite feature space. The direct involvement of the feature space in the estimation of the Koopman operator leads to rather complicated numerical machinery. To avoid directly computing the infinite dimensional vectors that result, an involved collection of linear algebra techniques are leveraged to extract the Koopman modes. Here it is shown that this process may be simplified and that a procedure that directly involves the kernel functions centered at the snapshots simplifies the design of DMD algorithms. This approach keeps to the spirit of the ``kernel trick,'' where feature vectors are never  directly evaluated and only accessed through evaluations of the kernel function itself.

Throughout this algorithm, a Koopman operator will be assumed to be densely defined, as Section \ref{sec:misconceptions} demonstrated that most Koopman operators cannot be expected to be bounded or compact. An additional assumption will be made that the kernel functions themselves reside in the domain of the Koopman operator. It should be noted that since the kernels are always in the domain of the adjoint of the Koopman operator (see Section \ref{sec:koopmanoperators}), a finite rank representation of the adjoint of the Koopman operator may thus be derived without assuming that the kernels are in the domain of the Koopman operator.

For the sake of the derivation, it is also assumed that the Koopman operator is diagonalizable, which is not generally expected to be true. However, the finite rank representations leveraged in this manuscript are almost always diagonalizable, since the set of non-diagonalizable matrices are of measure zero in the collection of all matrices. Moreover, for periodic data sets, the adjoint of the Koopman operator will be invariant on the span of the collection of kernel functions centered at the snapshots, and thus, the finite rank representations will be explicitly the adjoint of the Koopman operator on that subspace, which supports the assumption of the availability of eigendecompositions for the Koopman operator in the periodic or quasiperiodic settings.

For a given collection of snapshots $\{ x_1,x_2,...,x_m \}$\footnote{While availability of a time series of snapshots $\{ x_1,x_2,...,x_m \}$ such that $ x_{i+1} = F(x_i) $ is a more typical use case, the developed method does not require such a time series. It can also be implemented using arbitrary snapshots $\{ x_1,x_2,...,x_m \}$ and $\{ y_1,y_2,...,y_m \}$ provided $ y_i = F(x_i)$. }, the goal is to determine a finite rank representation of $\mathcal{K}_F$ that is derived from the kernel functions centered at the snapshots. To express a finite rank representation, the ordered basis $\alpha = \{ k_{x_1},...,k_{x_{m-1}} \}$ is leveraged. In particular, if $P_{\alpha}$ is the projection on to $\vspan(\alpha)$, the operator $P_\alpha \mathcal{K}_{F}$ maps $\vspan(\alpha)$ to itself, which enables the discussion of eigenfunctions and eigenvalues of $\mathcal{K}_{F}$ using only functions in $\vspan(\alpha)$.

Suppose that given a function $g\in\vspan{\alpha}$, with coefficients $ a_1,...,a_{m-1}$, the function $P_\alpha \mathcal{K}_{F}g$ can be expressed in the basis $\alpha$ using the coefficients $b_1,...,b_{m-1}$. As a result, the operator $P_\alpha \mathcal{K}_{F}$ can be represented using a matrix that maps the vector $\left(a_1,...,a_{m-1}\right)^T$ to the vector $\left(b_1,...,b_{m-1}\right)^T$. In the following development, the finite rank representation of $P_\alpha \mathcal{K}_{F}$, expressed in a matrix form, is denoted by $[P_{\alpha} \mathcal{K}_F]_\alpha^\alpha$, where the notation $[\cdot]_\alpha^\alpha$ indicates that both the domain and range of $P_\alpha \mathcal{K}_{F}$ is restricted to $\vspan(\alpha)$.

Consequently, the $i$th column of $[P_{\alpha} \mathcal{K}_F]_\alpha^\alpha$ may be determined through the examination of the action of the operator $P_{\alpha} \mathcal{K}_F$ on the basis function $K_{x_i}$, for
$i = 1,\ldots,m-1.$ Using the fact that $\mathcal{K}_F K_{x_i}(x) = K_{x_{i}}(F(x))$, $[P_{\alpha} \mathcal{K}_F]_\alpha^\alpha$ may be written as
\begin{gather*}
[P_{\alpha} \mathcal{K}_F]_\alpha^\alpha = \begin{pmatrix} K(x_1,x_1) & \cdots & K(x_1,x_{m-1})\\ \vdots & \ddots & \vdots\\ K(x_{m-1},x_1) & \cdots & K(x_{m-1},x_{m-1})\end{pmatrix}^{-1}
\begin{pmatrix} K(x_2,x_1) & \cdots & K(x_2,x_{m-1})\\ \vdots & \ddots & \vdots\\ K(x_{m},x_1) & \cdots & K(x_{m},x_{m-1})\end{pmatrix}.
\end{gather*}
It should be noted that this is precisely the pair of matrices examined in \citep{williams2014kernel} after the use of a truncated SVD.

The objective of DMD is to use the finite rank representation determined above to create a data driven model of the dynamical system. This makes use of a fundamental property of eigenfunctions of the Koopman operator. In particular, suppose that $\varphi$ is an eigenfunction of $\mathcal{K}_F$ with eigenvalue $\lambda$. Evaluating the eigenfunction at a snapshot reveals $\varphi(x_{i+1}) = \varphi(F(x_i)) = \mathcal{K}_F \varphi(x_i)=\lambda \varphi(x_i)$. Hence, $\varphi(x_{i+1}) = \lambda^i \varphi(x_1)$. Now suppose that $\{ \varphi_{j} \}_{j=1}^\infty$ is an eigenbasis for a diagonalizable Koopman operator, $\mathcal{K}_F$, corresponding to the eigenvalues $\{\lambda_j\}_{j=1}^\infty$. For a state $x \in \mathbb{R}^n$, let $(x)_d$ be the $d$-th component of $x$ for $d=1,\ldots,n$. If it is assumed that the mapping $x \mapsto (x)_d$ is in the RKHS (as it is when $H$ is the native space for the exponential dot product space \citep{steinwart2008support}), then it may be expressed as $(x)_d = \lim_{M \to \infty} \sum_{j=1}^M (\xi_{j,M})_{d} \varphi_i(x)$ for some coefficients $\left\{(\xi_{j,M})_{d}\right\}_{j=1}^{\infty}$. Note that since the Koopman operator is not generally a normal operator, $\{ \varphi_i \}_{i=1}^\infty$ is not expected to be an orthonormal basis, and hence, there may be nonzero influences between the coefficients obtained by projection and this is expressed by the additional index $M$ in $\xi_{j,M}$. This means that a series representation of the decomposition as expressed in \cite{kawahara2016dynamic,brunton2019data} is not always possible. \textit{Hence, Koopman modes are not fixed quantities unless there is an orthonormal basis of eigenfunctions for the Koopman operator.} By stacking each $(x)_d$, the full state observable $g_{id}$, given by $g_{id}(x) = x$, is expressed as
\begin{equation}g_{id}(x)=\lim_{M\to\infty} \sum_{j=1}^M \xi_{j,M}\varphi_j(x).\label{eq:gidProjection}\end{equation} Hence, each snapshot may be reconstructed as
\begin{equation}x_{i+1} = \lim_{M\to\infty} \sum_{j=1}^M \xi_{j,M} \lambda_j^i\varphi_j(x_1).\end{equation}

Since the Koopman operator is approximated here by a finite rank representation, perfect reproduction of $g_{id}$ through a series of eigenfunctions is not possible. Instead, eigenfunctions determined through the finite rank representation are used to construct the approximation of $g_{id}$. In particular, the matrix $[P_{\alpha} \mathcal{K}_F]_{\alpha}^\alpha$ is the matrix representation of $P_{\alpha} \mathcal{K}_F$. If $v_j$ is an eigenvector for the matrix $[P_{\alpha} \mathcal{K}_F]_{\alpha}^\alpha$ with eigenvalue $\lambda_j$, then
\begin{gather*}
    P_{\alpha} \mathcal{K}_F \left(\sum_{i=1}^{m-1} (v_j)_i K(x,x_i)\right) = \begin{pmatrix} K(x,x_1)\\ \vdots\\ K(x,x_{m-1}) \end{pmatrix}^T [P_{\alpha} \mathcal{K}_F]_{\alpha}^\alpha v_j = \lambda_j\sum_{i=1}^{m-1} (v_j)_i K(x,x_i).
\end{gather*}
That is, $\sum_{i=1}^{m-1} (v_j)_i K(x,x_i)$ is an eigenfunction of $P_{\alpha} \mathcal{K}_F$. The corresponding normalized eigenfunction is denoted by $\hat \varphi_j(x) := \frac{1}{\sqrt{v_j^T G v_j}} \sum_{i=1}^{m-1} (v_j)_i K(x,x_i),$
where $G = (K(x_i,x_\ell))_{i,\ell=1}^{m-1}$ is the Gram matrix associated with the snapshots and kernel function.

Using a finite rank representation of \eqref{eq:gidProjection}, it is easy to see that the $d$-th row of the matrix $\hat \xi := \begin{pmatrix}\hat{\xi}_1&\ldots&\hat{\xi}_{m-1}\end{pmatrix}$ of Koopman modes is comprised of the components of $ (x)_d $ along the (non-orthogonal) directions $\hat \varphi_j$. That is, $g_{id}(x_i) = x_i = \sum_{j=1}^{m-1} \xi_j \hat \varphi_j(x_i),$
which yields $\hat \xi = X\left(V^{T}G\right)^{-1},$
where $ X := \begin{pmatrix}x_1 & \ldots & x_{m-1}\end{pmatrix} $ is the data matrix and $$ V := \begin{pmatrix}\frac{v_1}{\sqrt{v_1^T G v_1}} & \ldots & \frac{v_{m-1}}{\sqrt{v_{m-1}^T G v_{m-1}}}\end{pmatrix} $$ is the matrix of normalized eigenvectors of $[P_{\alpha} \mathcal{K}_F]_{\alpha}^\alpha$.

Using the approximate eigenfunctions, $\hat \varphi_j$, a data driven model for the system is obtained as
\begin{equation}\label{eq:reconstruction}x_{i+1} \approx \sum_{j=1}^{m-1} \hat \xi_j  \lambda_j^i \hat\varphi_j(x_1).\end{equation}

\section{Vector Valued Considerations}

The attentive reader will notice that the ultimate objective of DMD is to achieve an approximation or decomposition of the function $g_{id}(x) := x$, the full state observable. However, the full state observable is $\mathbb{R}^n$ valued, whereas the RKHSs in question consist of scalar valued functions. Consequently, the coefficients that are determined to approximate the full state observable are vector valued. Up to now, the methods have simply separated the individual components of $g_{id}$ and established an approximation for each of them separately. The weights for these approximations are stacked to make the vector valued coefficients.

This section shows how the vector valued coefficients arise naturally from a projection onto vector valued functions in a vector valued RKHS. With the right selection of kernel operator, this projection operation reduces to the setting of Section \ref{sec:algorithm}.

A vector valued RKHS is a Hilbert space of functions from a set $X$ to a Hilbert space $\mathcal{Y}$, whose inner product will be denoted as $\langle \cdot, \cdot \rangle_{\mathcal{Y}}$, for which given any $v \in \mathcal{Y}$ and $x \in X$, the functional $g \mapsto \langle g(x), v \rangle_{\mathcal{Y}}$ is bounded. Hence, by the Reisz representation theorem, there exists a function $K_{x,v}$ in $H$ such that $\langle g, K_{x,v} \rangle_H = \langle g(x), v \rangle_{\mathcal{Y}}$ for all $g \in H$. It can be quickly seen that $K_{x,v}$ is linear in $v$, and hence, may be rewritten as $K_x v := K_{x,v}$, where $K_x : \mathcal{Y} \to H$.  The kernel operator $K(x,y)$ maps $(x, y) \in X \times X$ to an operator over $\mathcal{Y}$ as $K(x,y) := K_y^* K_x$. When $\mathcal{Y} = \mathbb{R}^n$, $K(x,y)$ may be taken to be a Hermitian matrix.

A Koopman operator over a vector valued RKHS behaves much like the Koopman operator in the scalar case. In particular, if $\mathcal{K}_F : \mathcal{D}(\mathcal{K}_F) \to H$ is a densely defined Koopman operator corresponding to the discrete dynamics $F:\mathbb{R}^n \to \mathbb{R}^n$ and $H$ is an $\mathbb{R}^n$ valued RKHS over $\mathbb{R}^n$, then $\mathcal{K}_F^* K_{x,v} = K_{F(x),v}$, which means that the Koopman operator advances the dynamics through the center of the kernel function.

Suppose that $\tilde K(x,y)$ corresponds to a scalar valued RKHS, $\tilde H$, and let $H$ be an $\mathbb{R}^n$ valued RKHS such that if $g \in H$, then $g = (g_1, \ldots, g_n)^T$ where $g_i \in \tilde H$ for each $i = 1,\ldots, n$ equipped with the inner product $\langle g, h \rangle_H = \sum_{i= 1}^n \langle g_i, h_i \rangle_{\tilde H}$. That $H$ is a vector valued RKHS follows since if $v \in \mathbb{R}^n$ and $x \in X$ then $|\langle g(x), v \rangle_{\mathbb{R}^n}| \le \| g(x) \|_{\mathbb{R}^n} \| v\|_{\mathbb{R}^n} = \sqrt{ \sum_{i=1}^\infty |\langle g_i,\tilde K(\cdot, x)\rangle_{\tilde{H}}|^2} \| v\|_{\mathbb{R}^n} \le \sqrt{\tilde K(x,x)} \sqrt{ \sum_{i=1}^\infty \| g_i \|_{\tilde H}^2} \| v\|_{\mathbb{R}^n} = \sqrt{\tilde K(x,x)} \| g\|_H \|v\|_{\mathbb{R}^n}$ , hence the functional $g \mapsto \langle g(x),v \rangle_{\mathbb{R}^n}$ is bounded. In this setting, $K(x,y) = \mathrm{diag}(\tilde K(x,y), \ldots, \tilde K(x,y))$, and $\langle K_{x} e_i, K_y e_j \rangle_{\mathbb{R}^n} = 0$ if $i\neq j$.

If $\{ x_0, \ldots, x_N \}$ is a collection of snapshots such that $F(x_i) = x_{i+1}$, then a finite rank approximation of $\mathcal{K}_F$ may be constructed by examining the image of the operator on $K_{x_i, e_j}$ for $i = 1,\ldots, N$ and $j=1,\ldots,n$, and then projecting back onto the span of these kernels. The projection operation requires the computation of the Gram matrix for the basis $\{ K_{x_i,e_j} \}$, which is a block diagonal matrix, where each block corresponds to a selection of dimension through $e_j$. Thus, if $\tilde G = ( \tilde K(x_s,x_\ell) )_{s,\ell = 1}^N$ the gram matrix manifests as

\[G = \begin{pmatrix} \tilde G & & \\ & \ddots & \\ & & \tilde G\end{pmatrix}.\]

Similarly, the interaction matrix is a block diagonal. Consequently, the vector of weights corresponding to each kernel function is composed of $n$ smaller vectors of length $N-1$, each one corresponding to a different dimension. Hence, each dimension may be treated independently. The eigenfunctions for this finite rank representation of the Koopman operator are then composed of $n$ identical collections of $N-1$ functions, differing only in the dimension they are supported in. Let $\tilde \varphi_1, \ldots, \tilde \varphi_N \in \tilde H$ be these scalar valued functions, and write $\varphi_{s,i} := \tilde \varphi_s e_i$ be the corresponding vector valued eigenfunction.

If the full state observable is projected onto this collection of eigenfunctions as $g_{id}(x) \approx \sum_{i=1}^n \sum_{s=1}^N w_{s,i} \phi_{s,i}(x) = \sum_{s=1}^N \tilde \phi_s(x) \left( \sum_{i=1}^n  w_{s,i} e_i \right).$ Here, $\sum_{i=1}^n w_{s,i} e_i = \xi_s$, where $\xi_s$ is the Koopman mode from the previous section.

\begin{figure}[ht]

    \begin{subfigure}[b]{0.45\textwidth}
    \centering

    \includegraphics[width=0.75in]{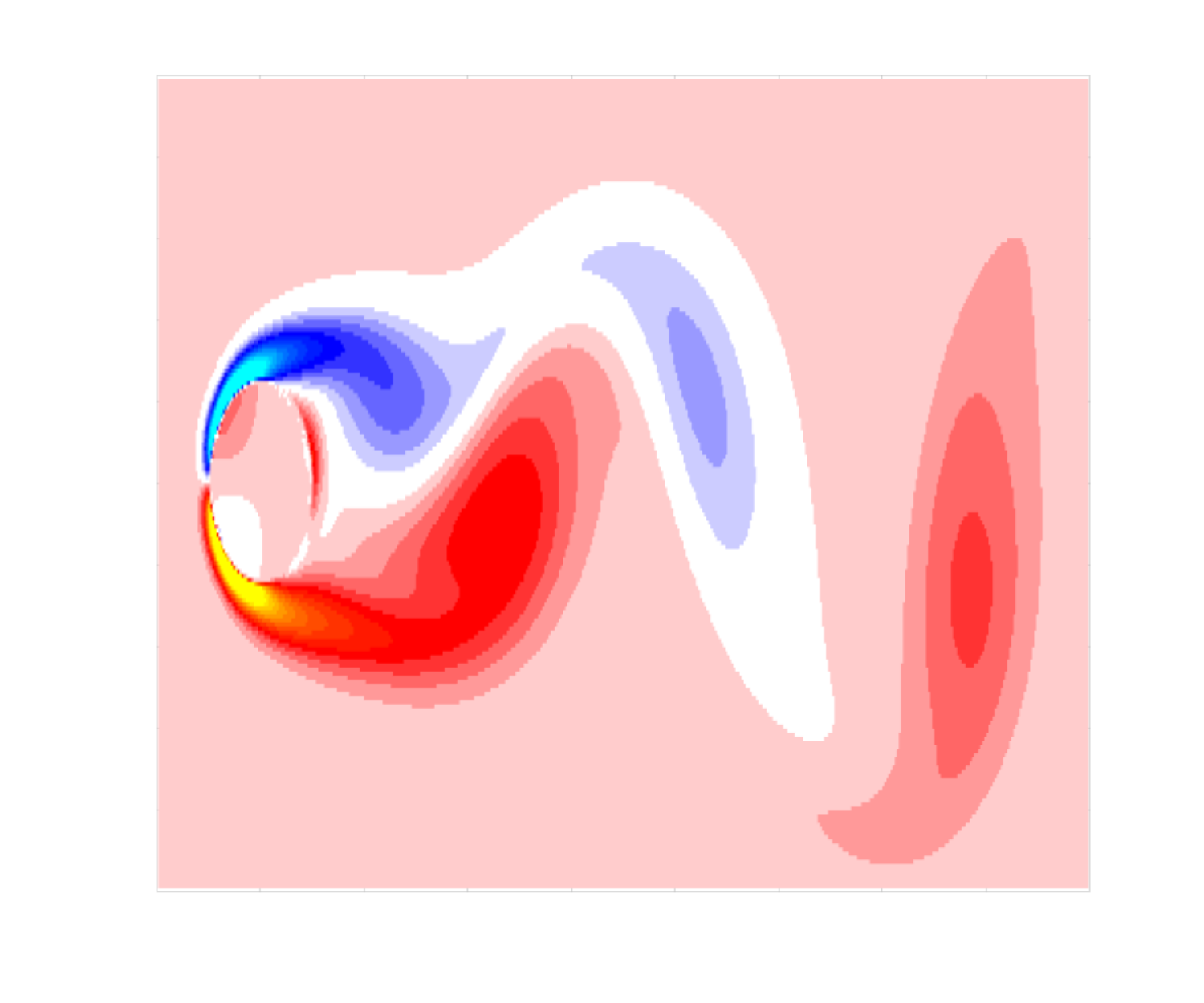}
    \includegraphics[width=0.75in]{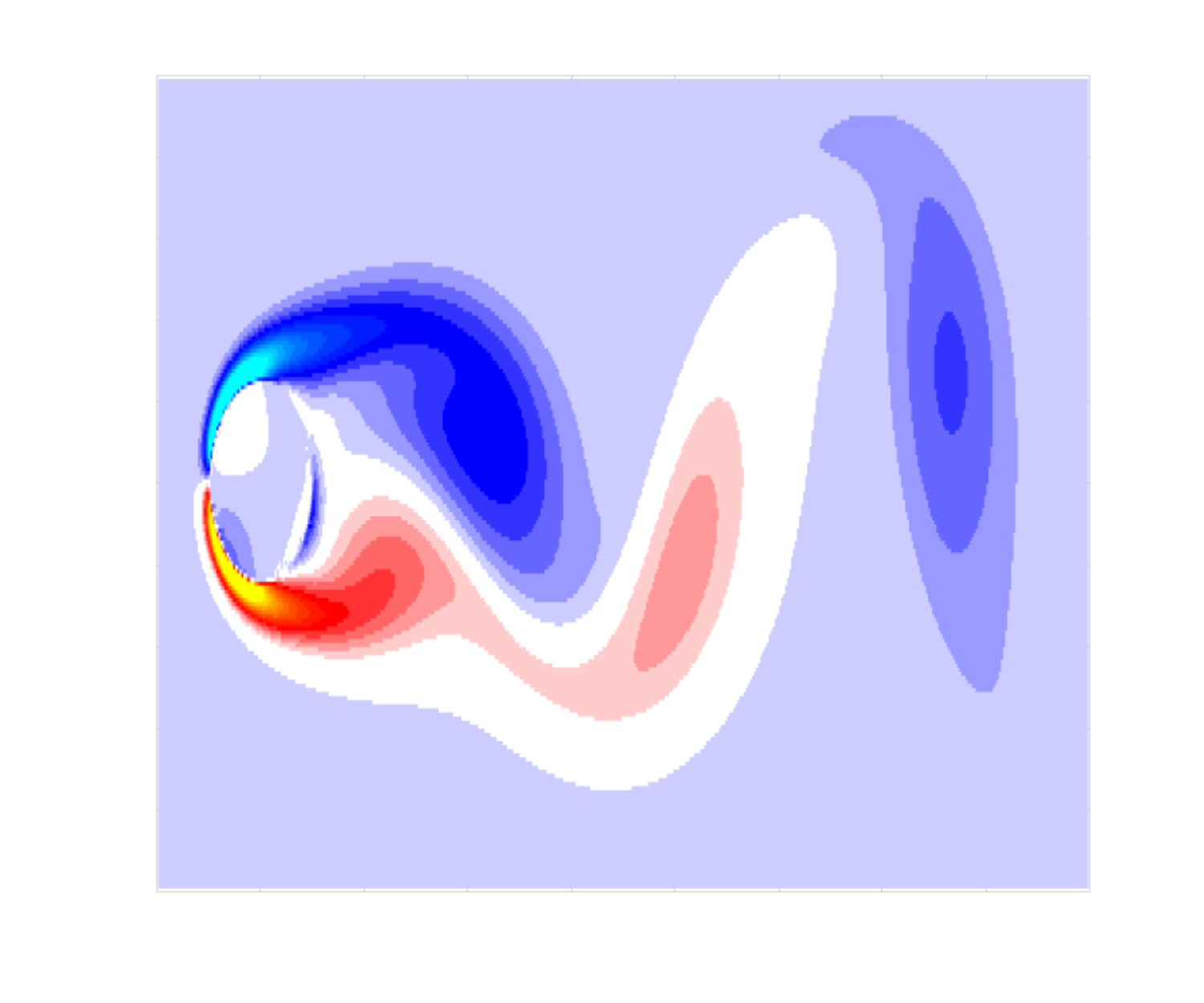}
    \includegraphics[width=0.75in]{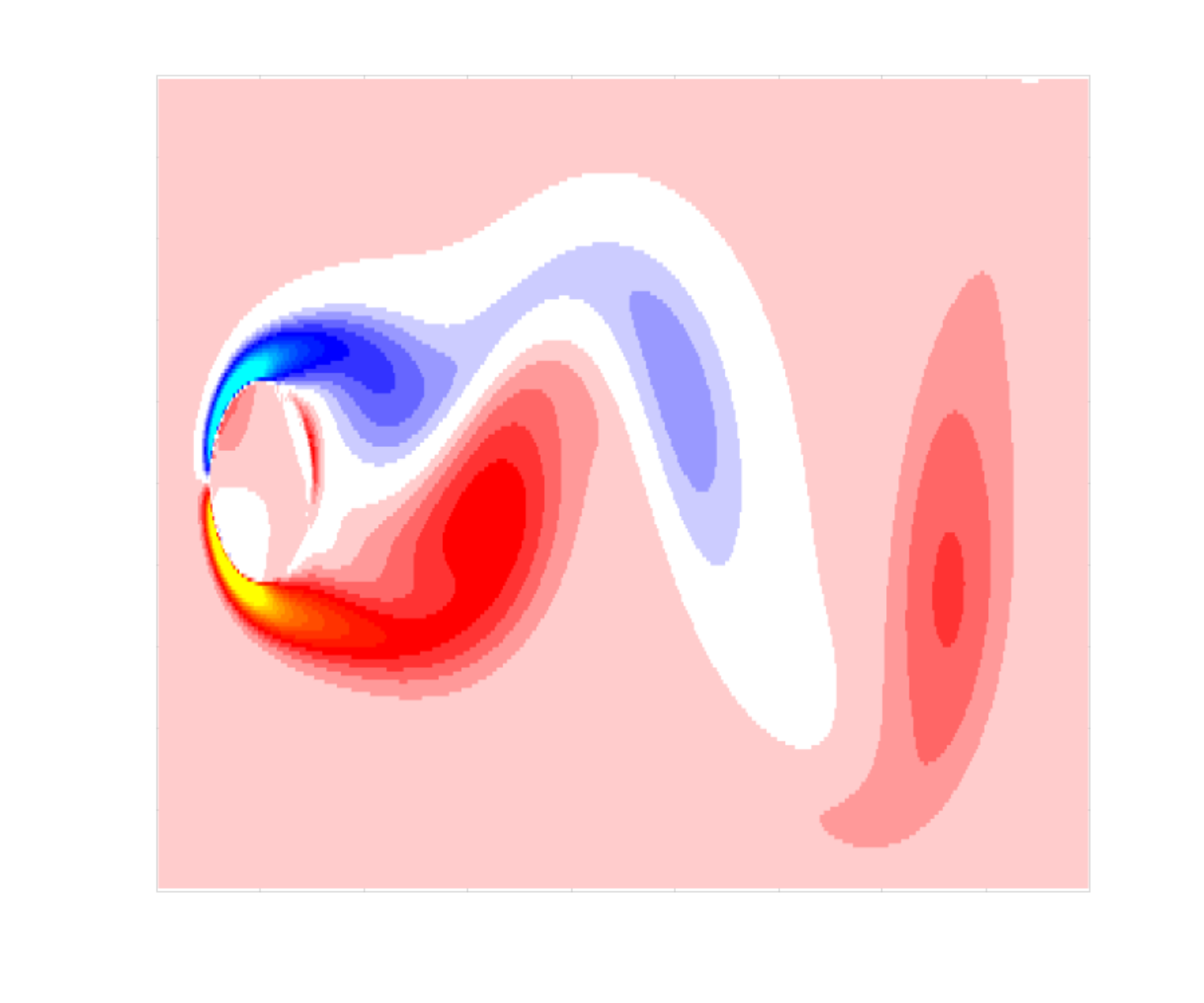}

    \includegraphics[width=0.75in]{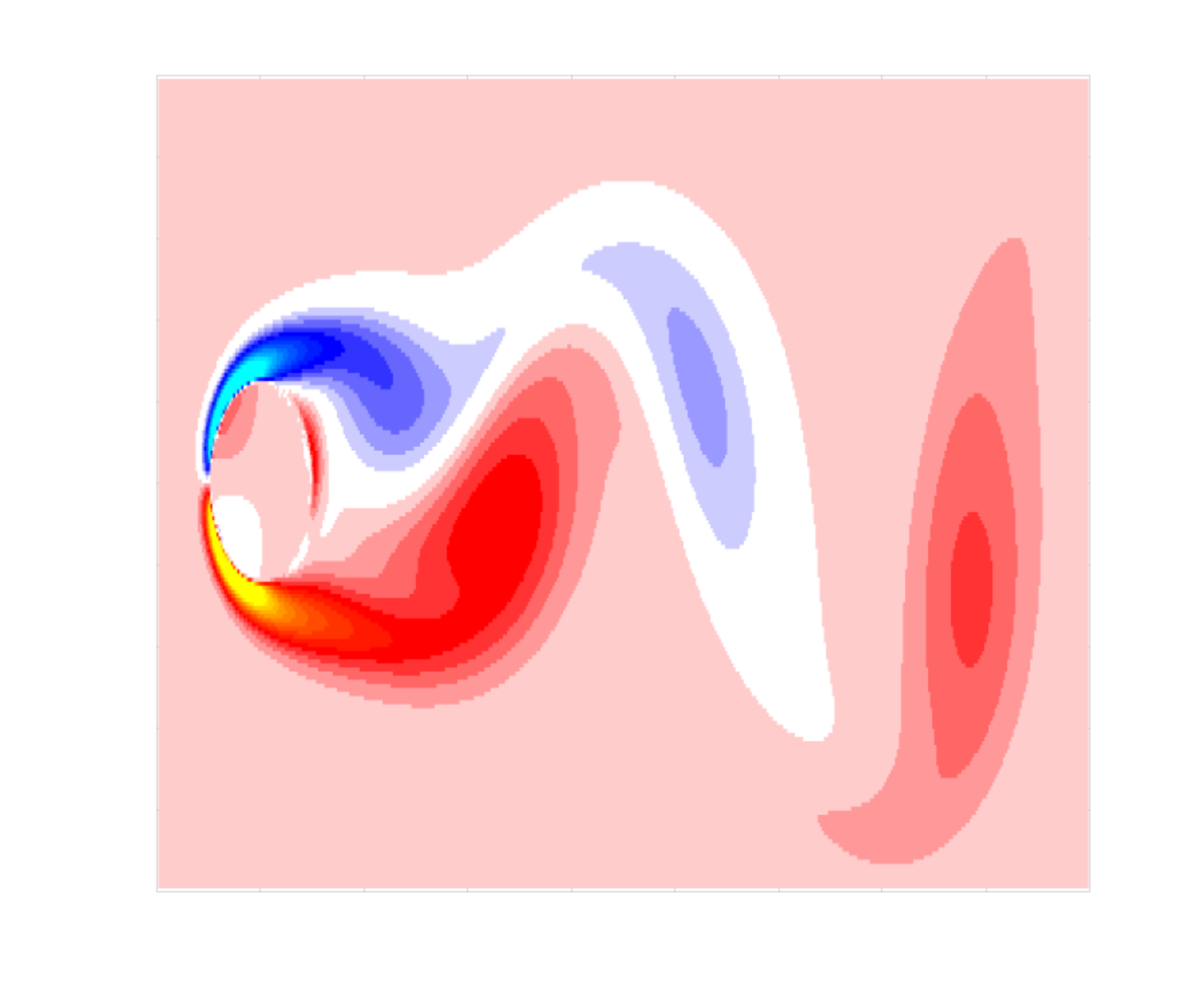}
    \includegraphics[width=0.75in]{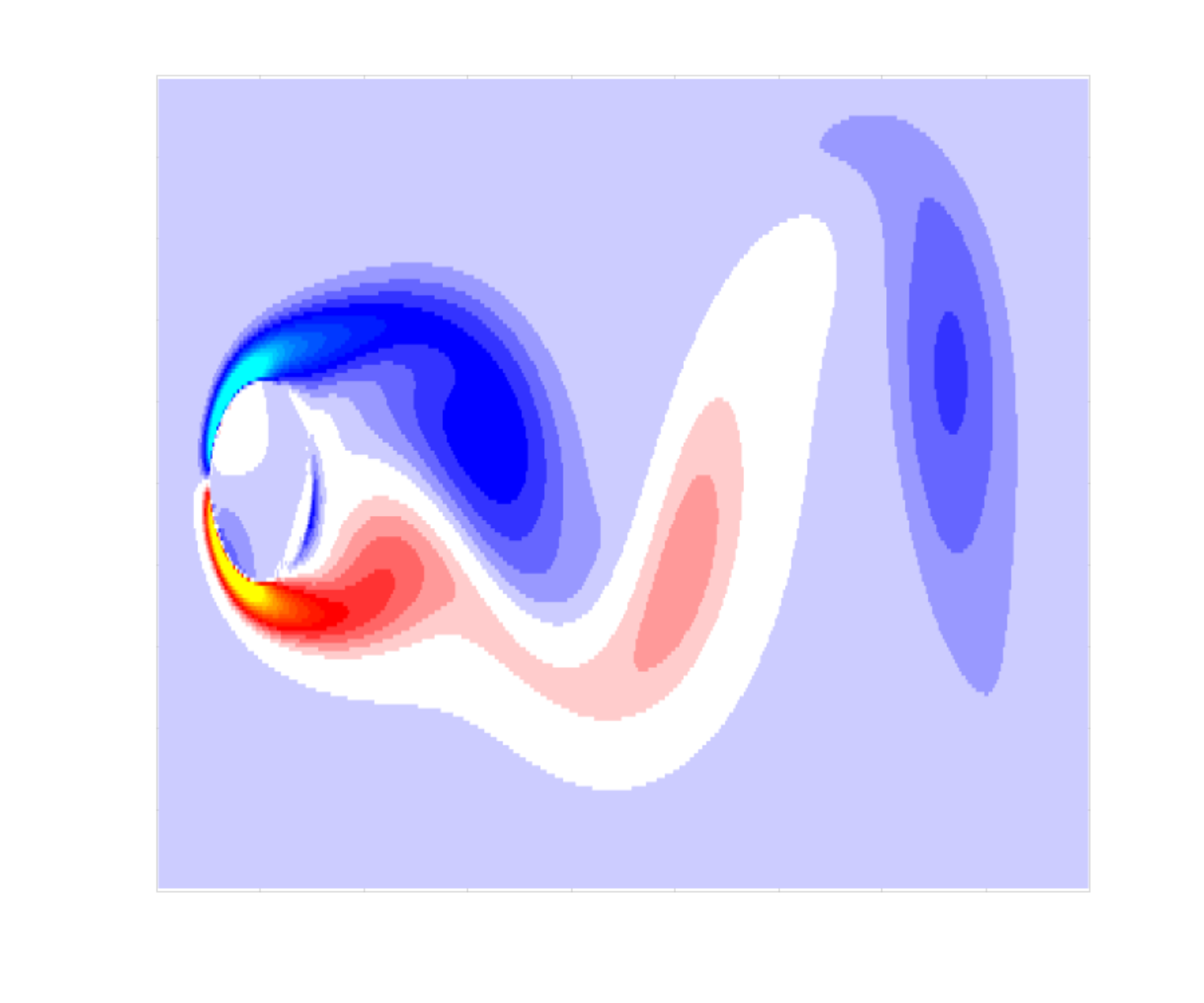}
    \includegraphics[width=0.75in]{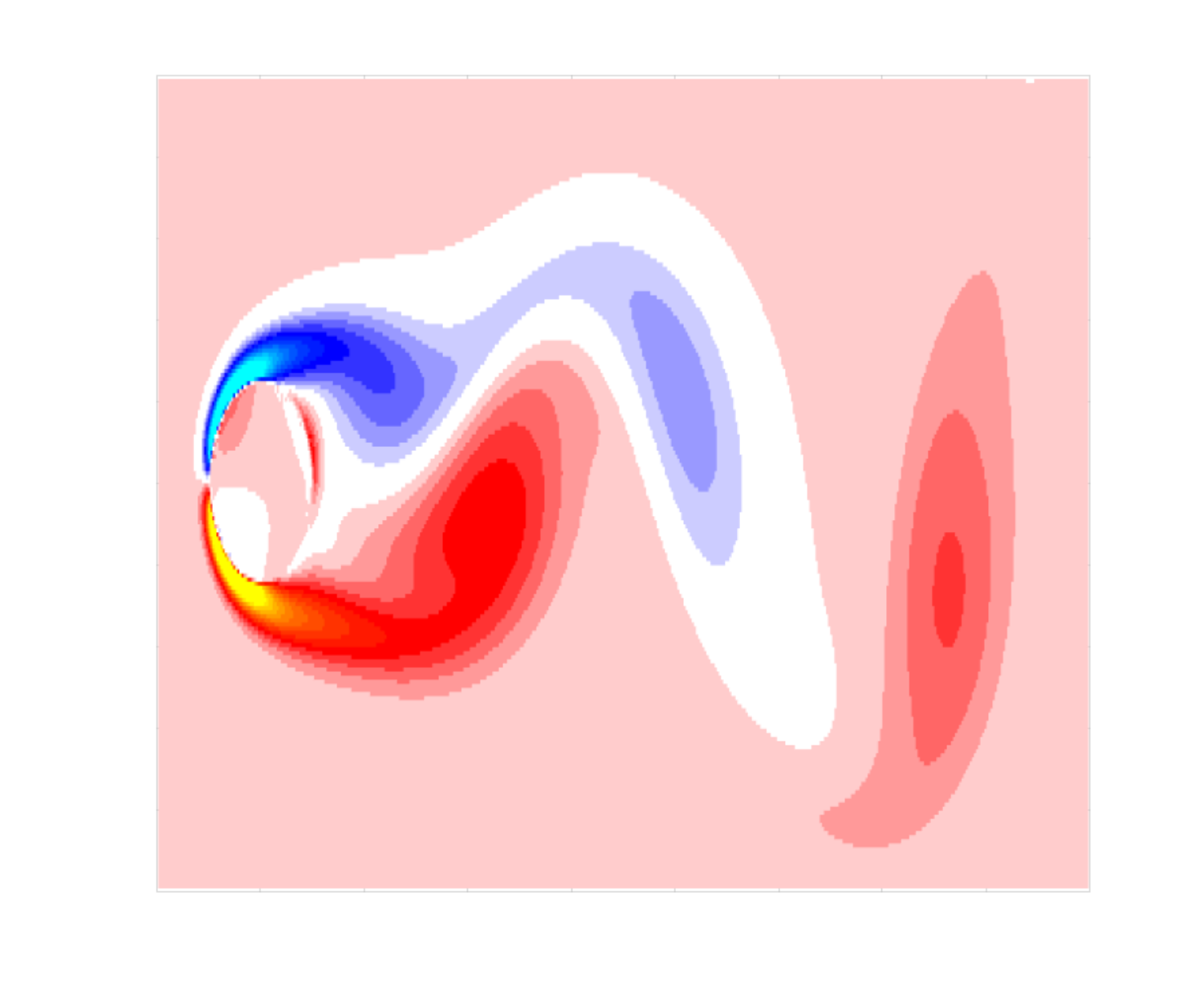}
    
    \includegraphics[width=0.75in]{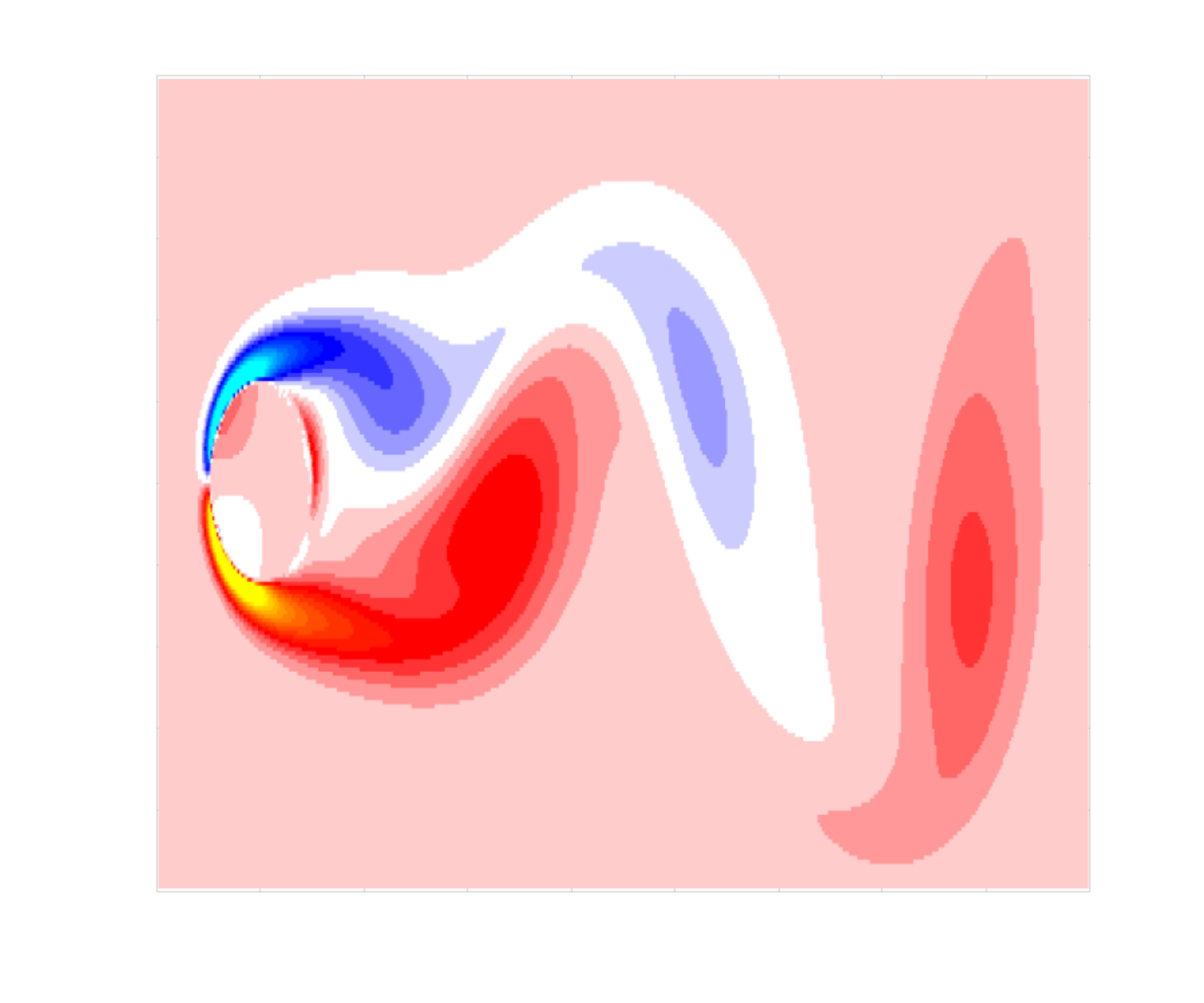}
    \includegraphics[width=0.75in]{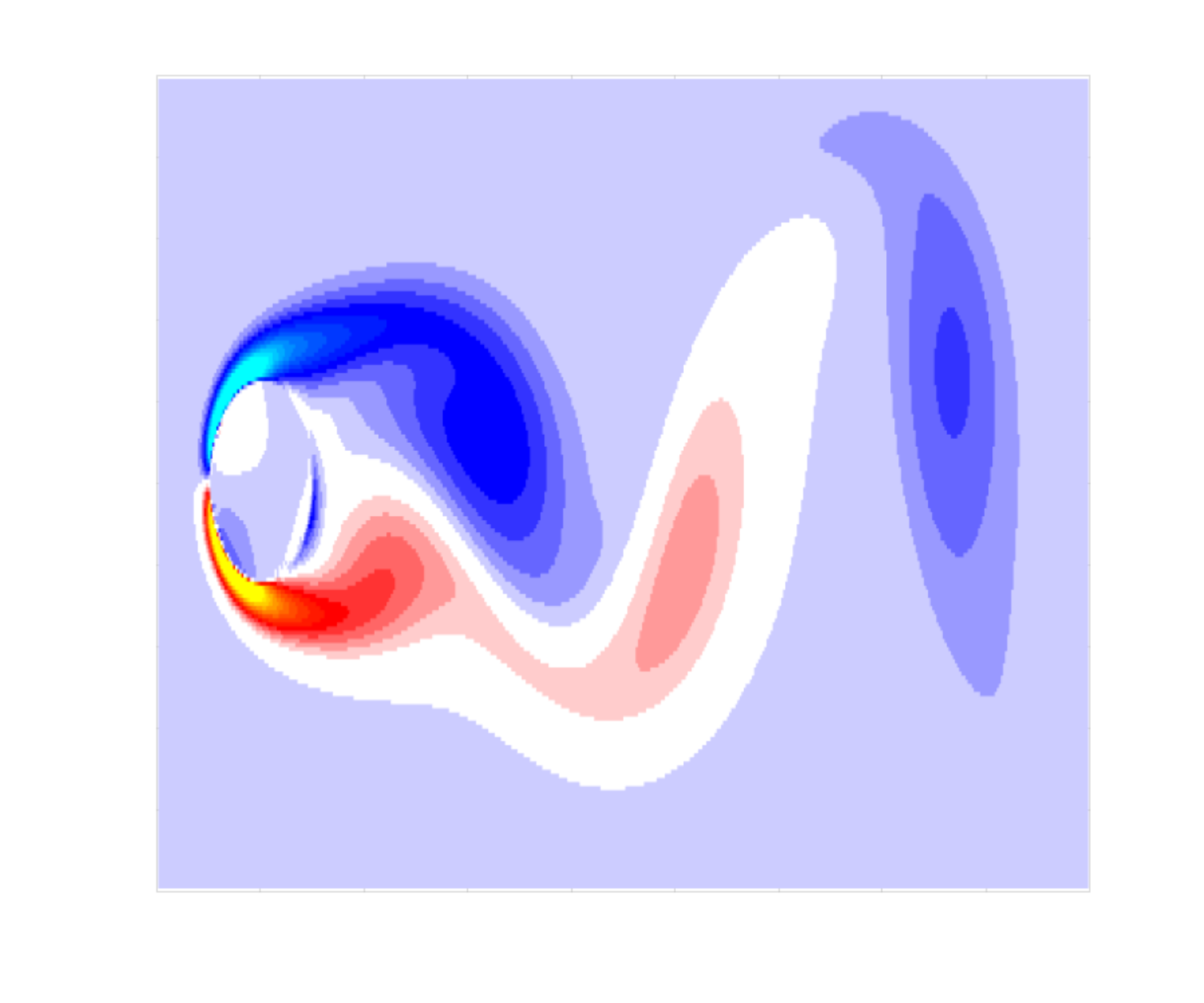}
    \includegraphics[width=0.75in]{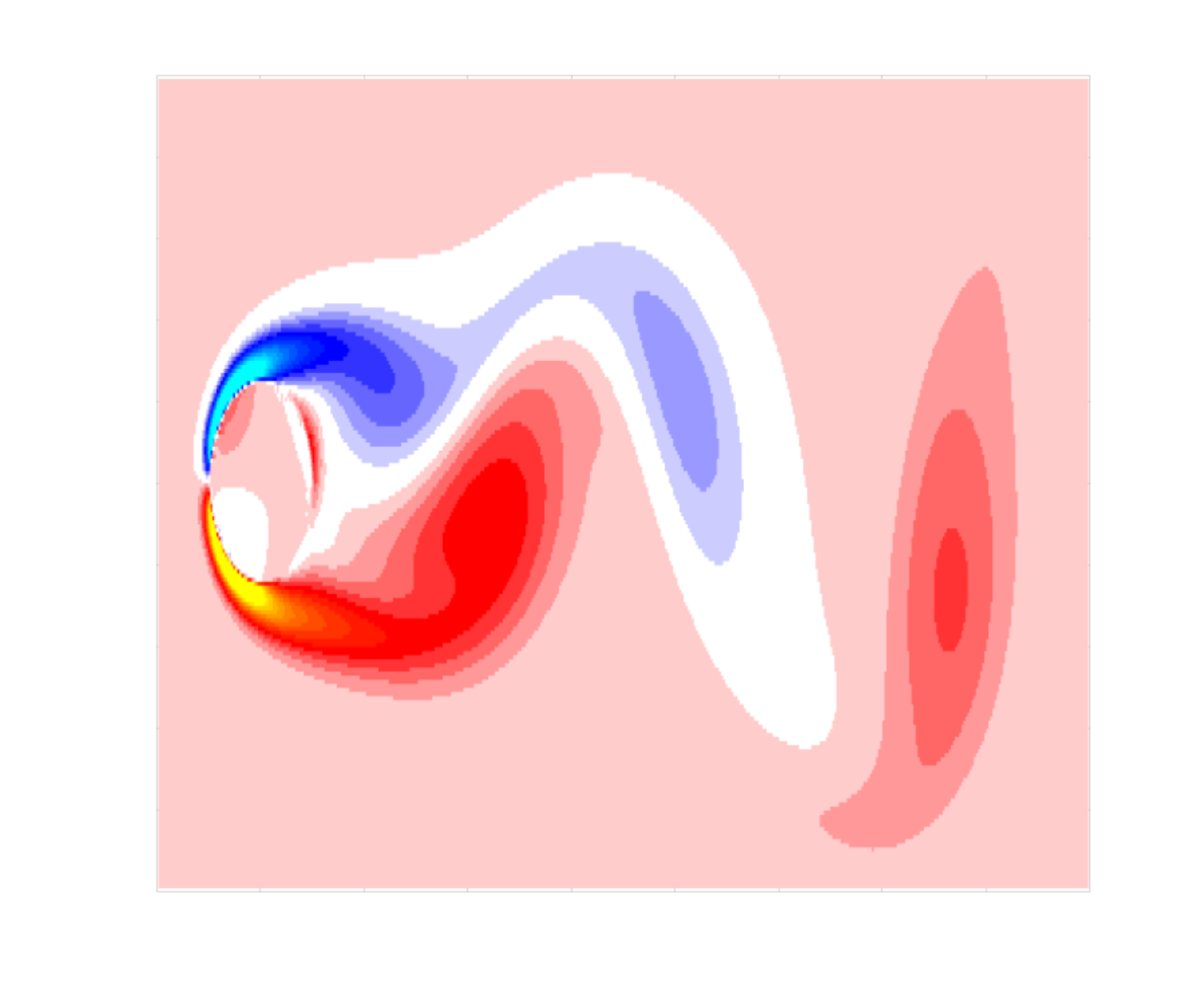}
    
    \caption{This figure presents the reconstruction of the vorticity field at the same time points using two different kernel functions. The left column presents the reconstruction of the initial state, $x_1$. The middle column shows the reconstructions of the state, $x_{15}$, and the right column corresponds to the reconstruction of the state $x_{30}$.}
    \label{fig:reconstruction}
    
    \end{subfigure}
\hfill
    \begin{subfigure}[b]{0.45\textwidth}
    \centering

    \includegraphics[width=0.75in]{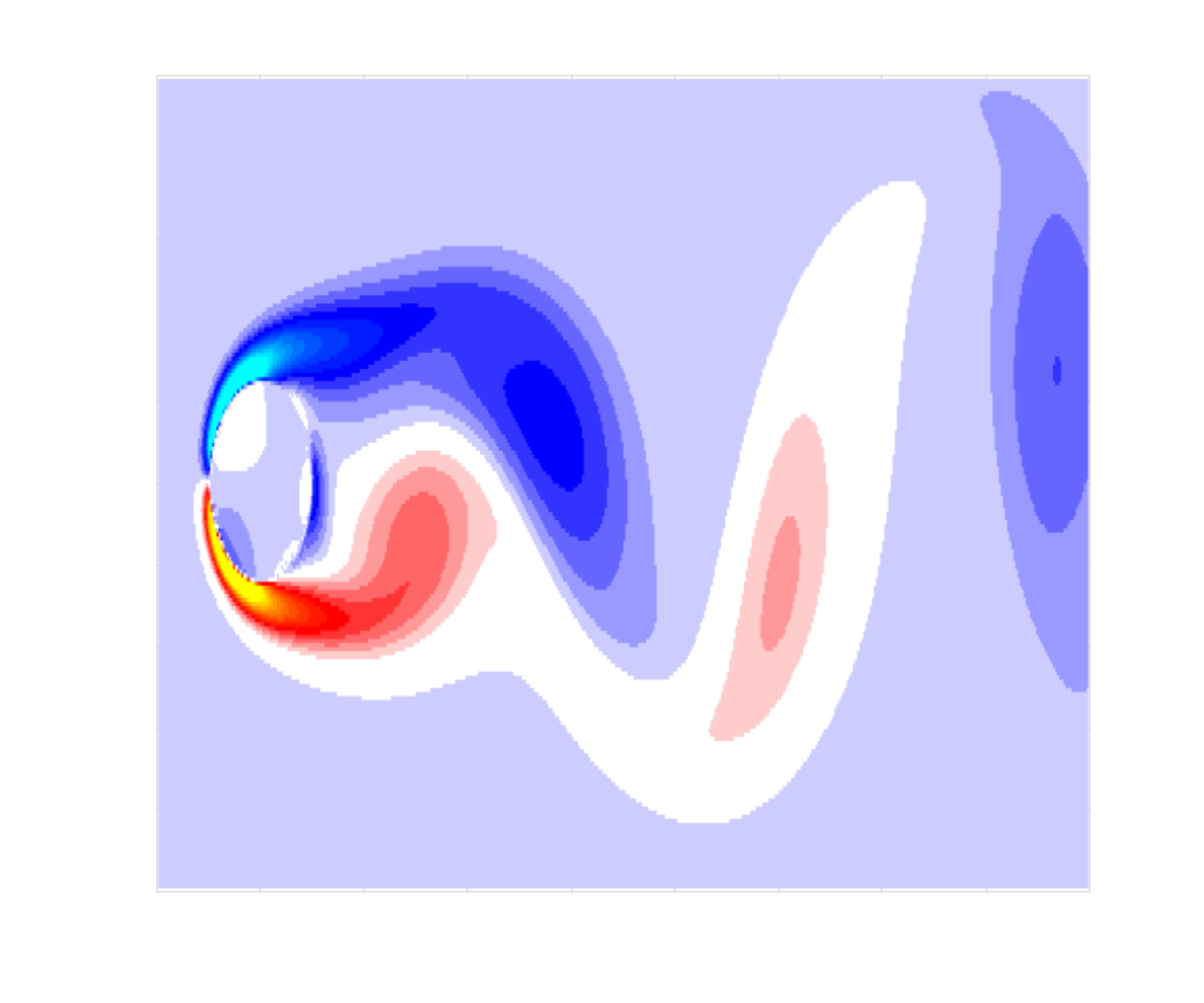}
    \includegraphics[width=0.75in]{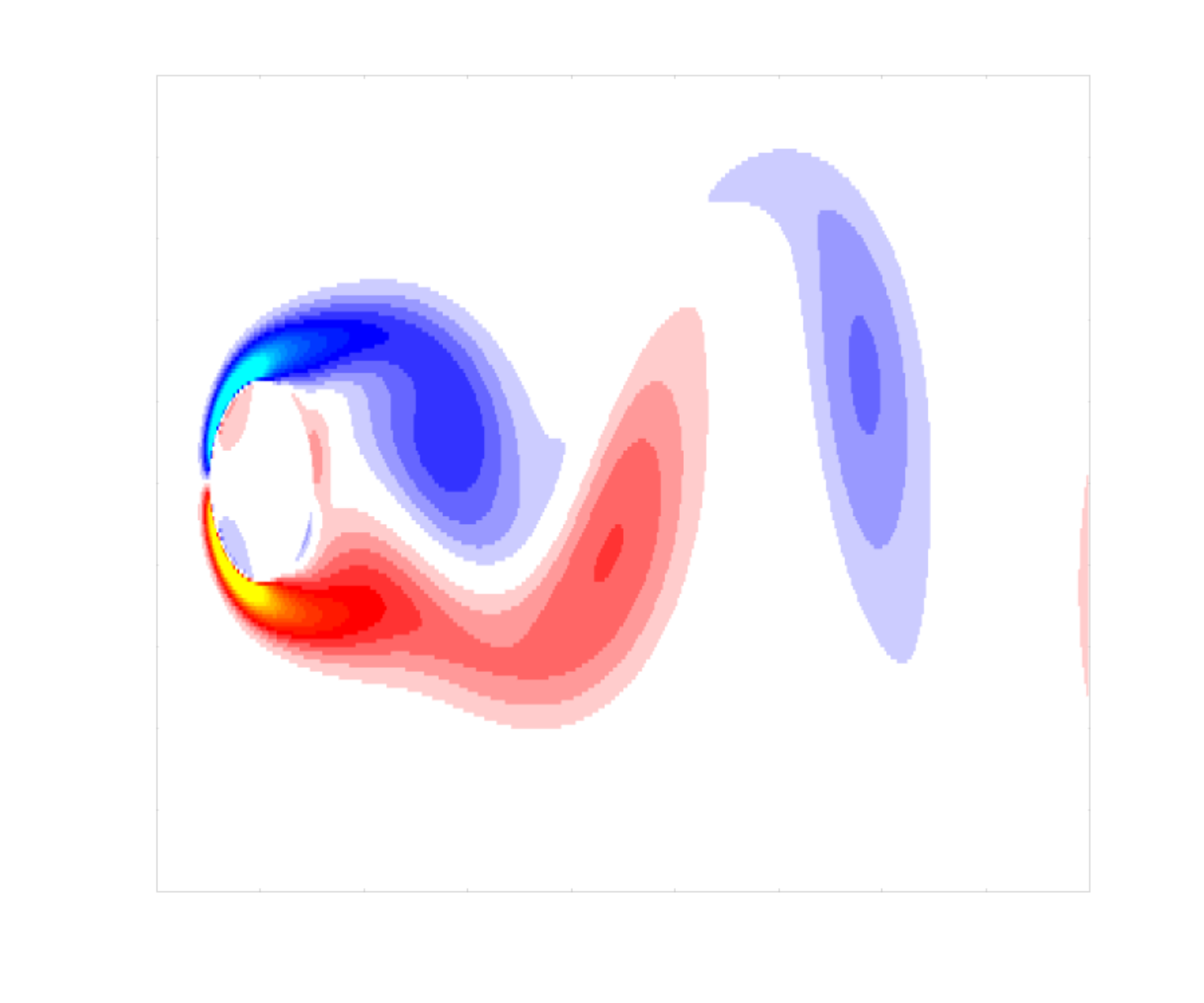}
    \includegraphics[width=0.75in]{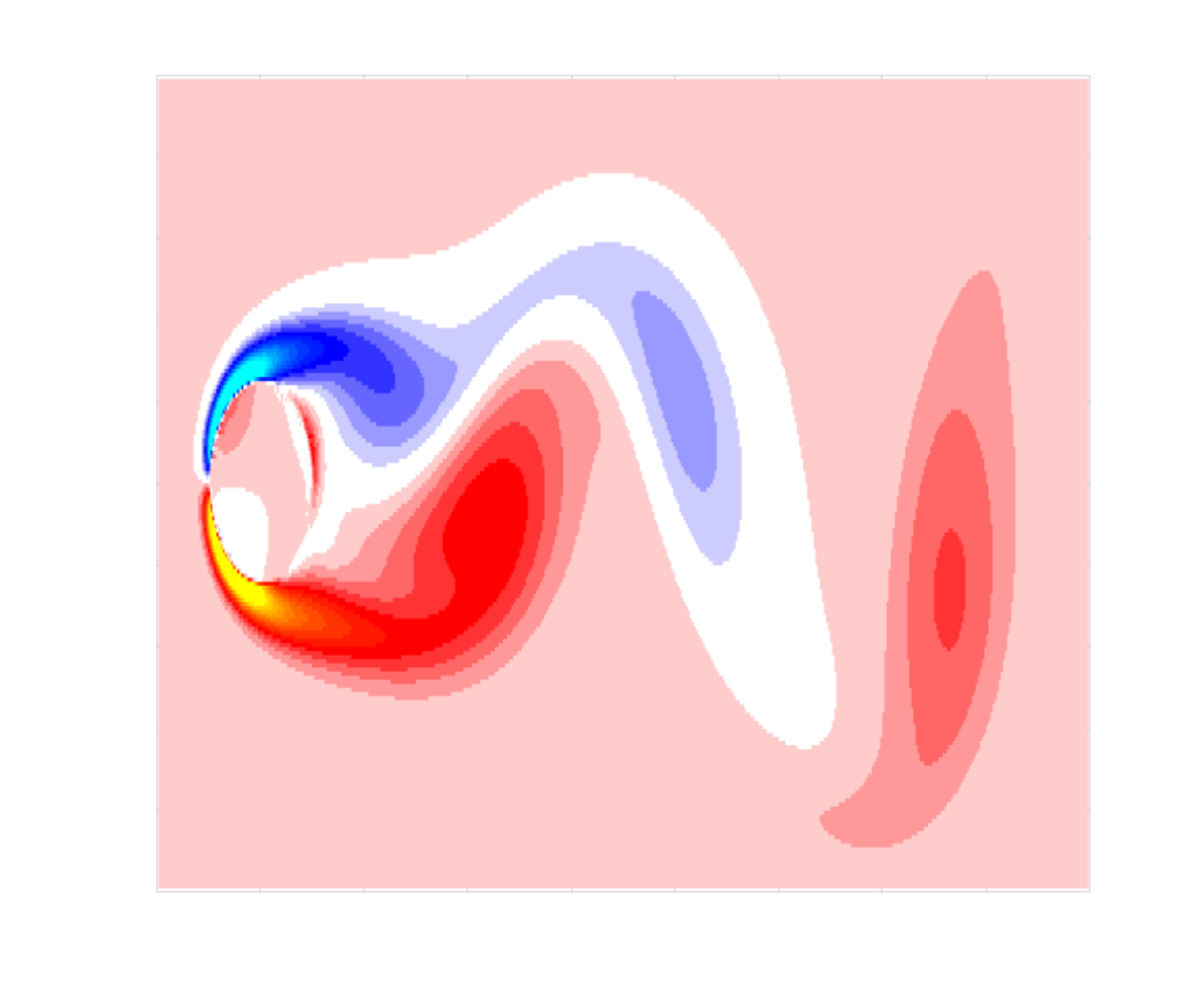}

    \includegraphics[width=0.75in]{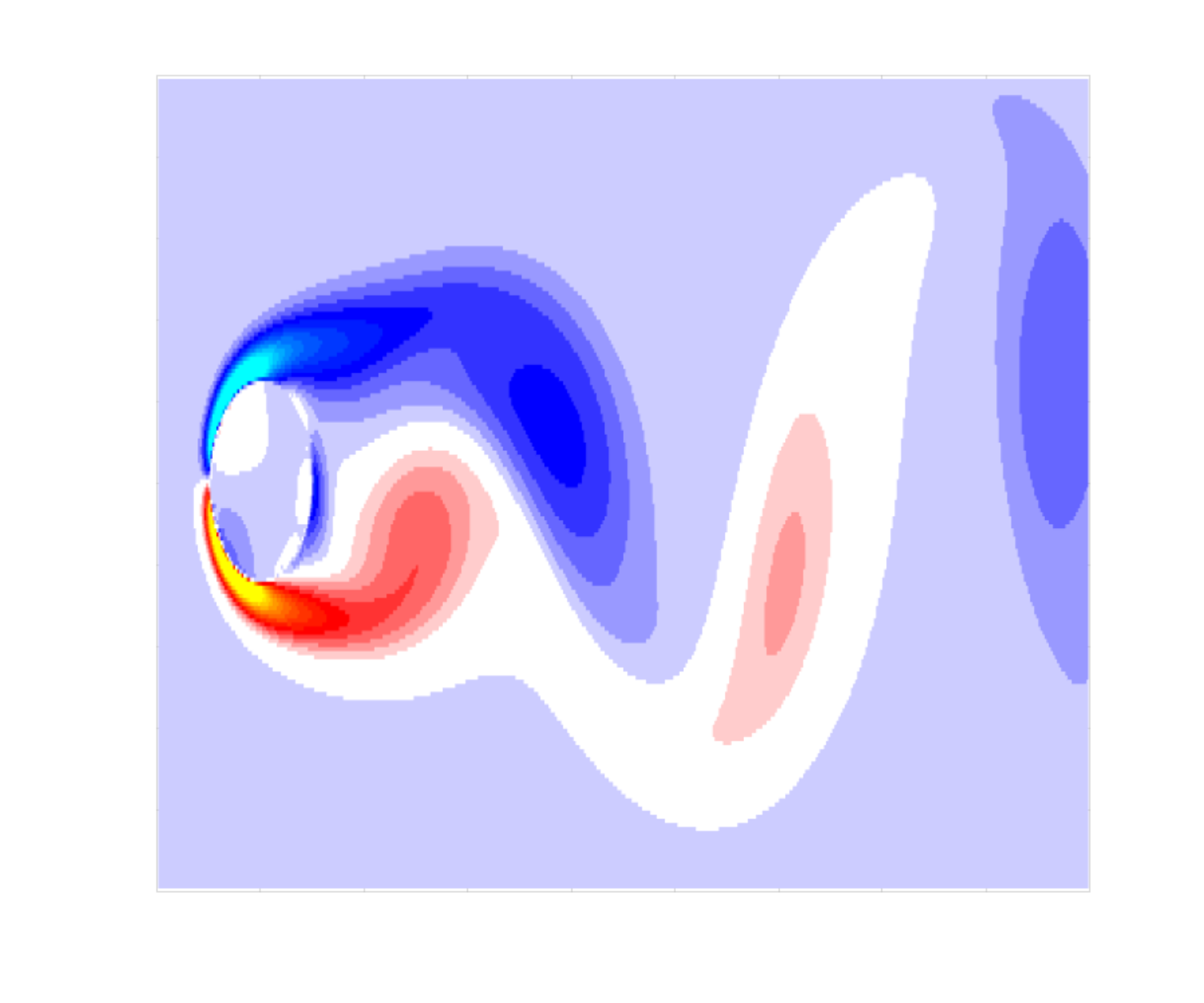}
    \includegraphics[width=0.75in]{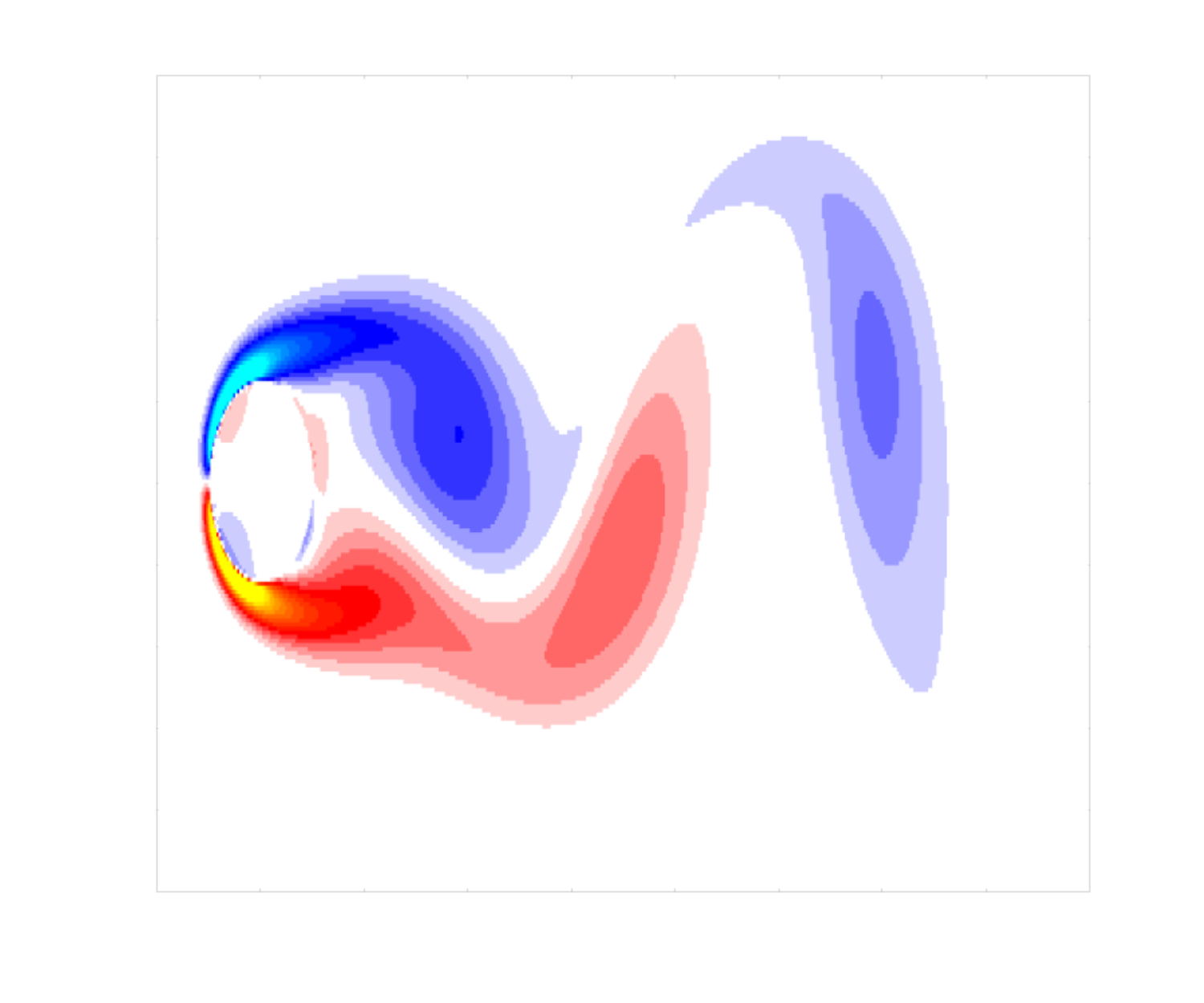}
    \includegraphics[width=0.75in]{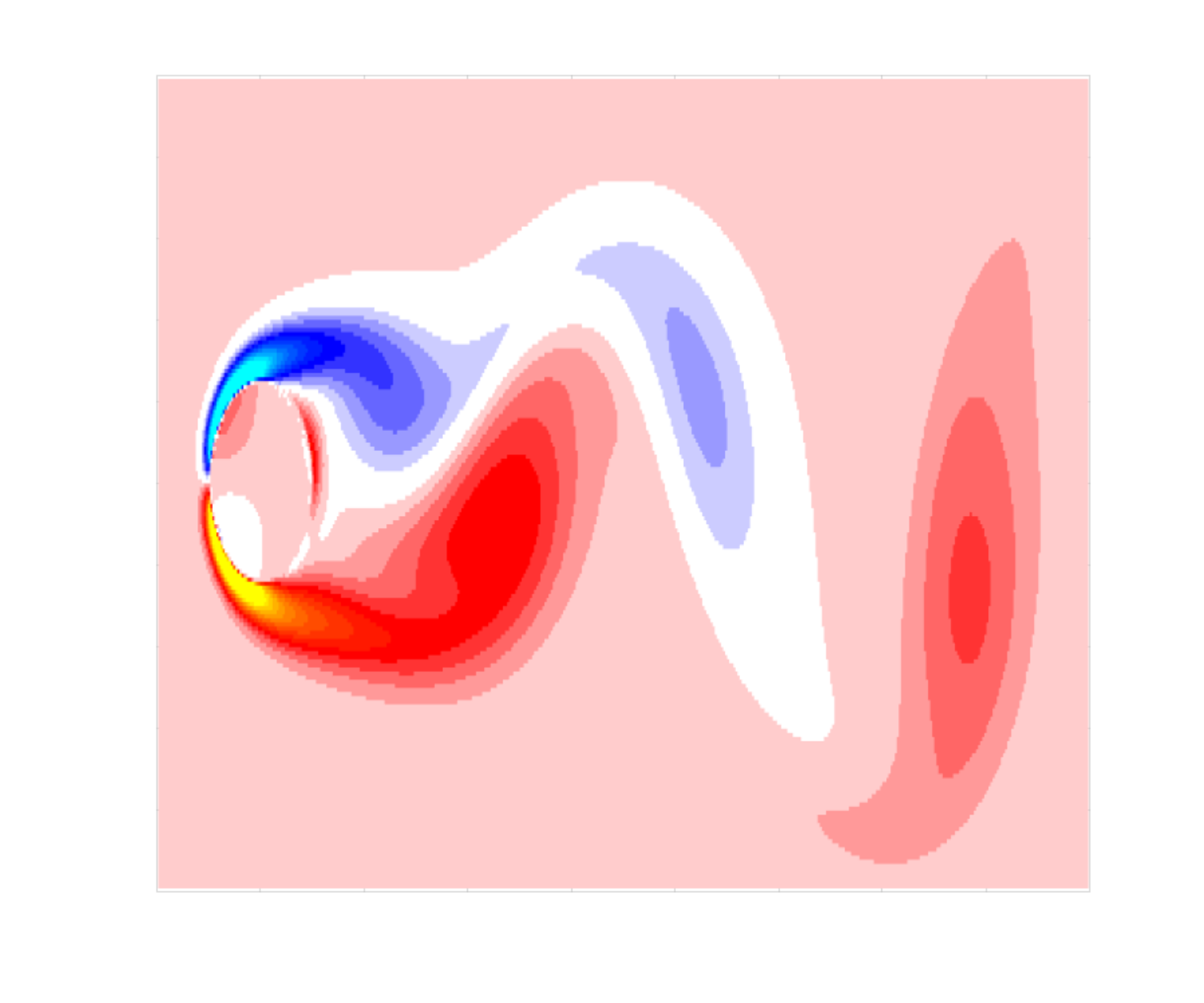}
    
    \includegraphics[width=0.75in]{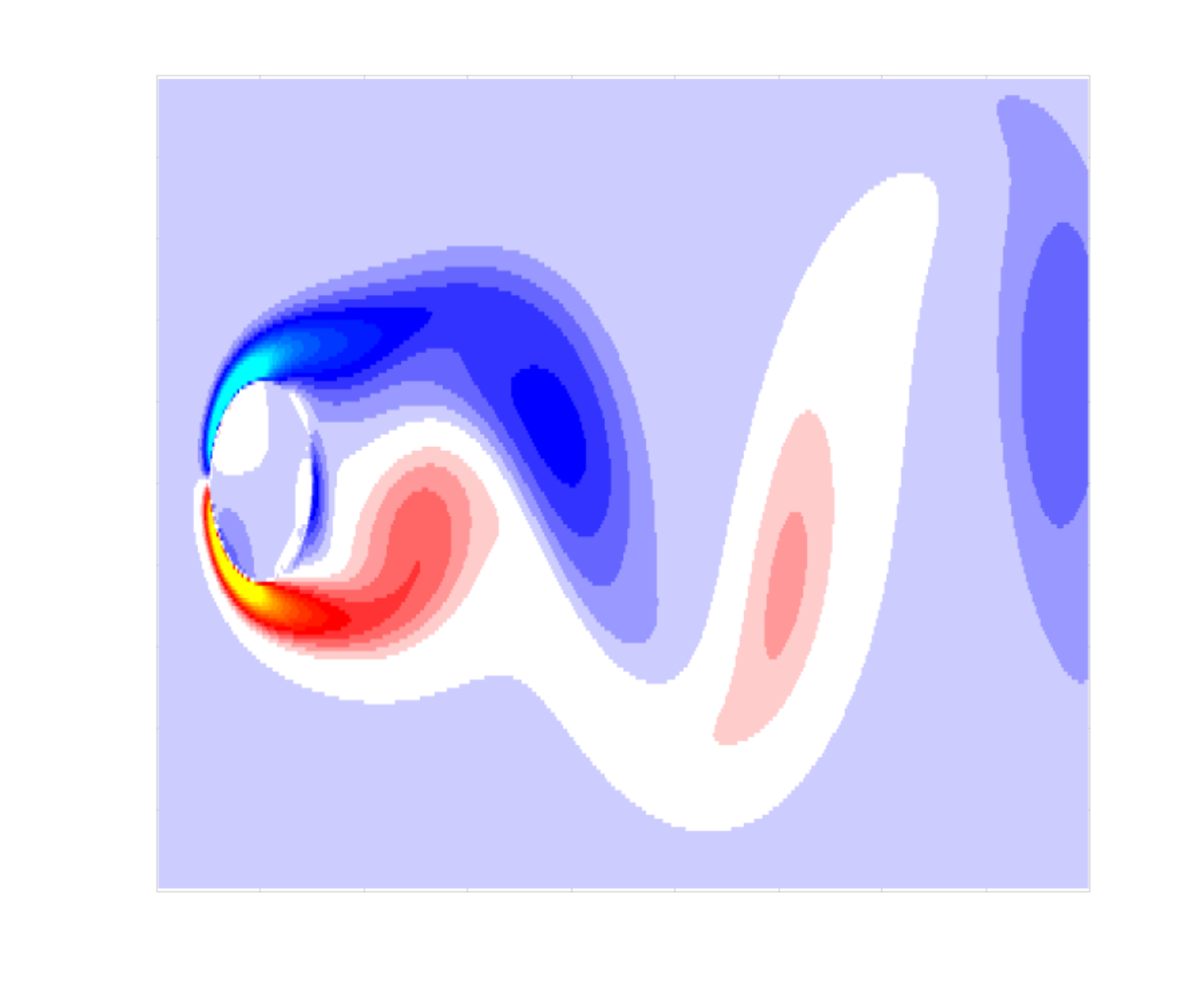}
    \includegraphics[width=0.75in]{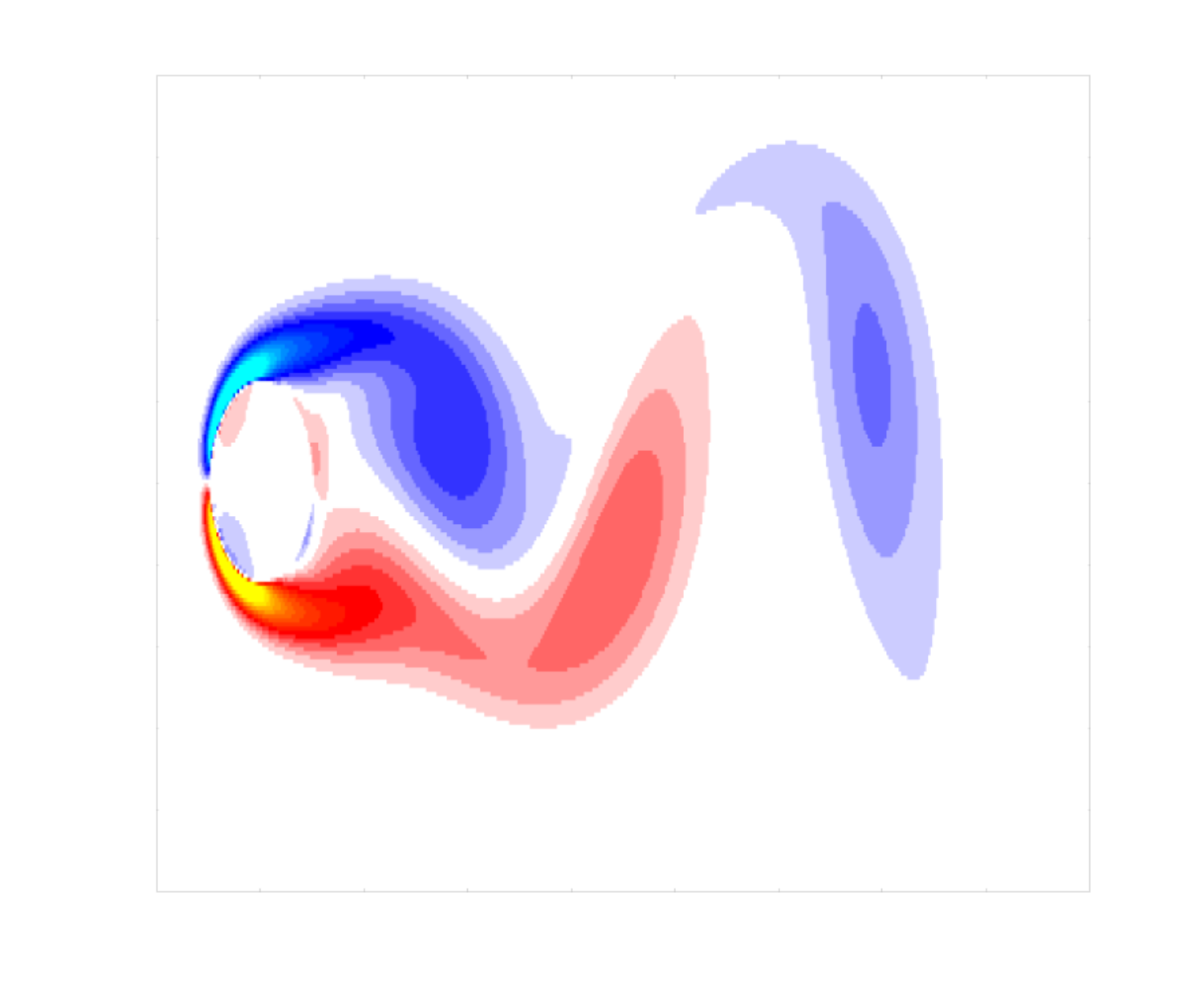}
    \includegraphics[width=0.75in]{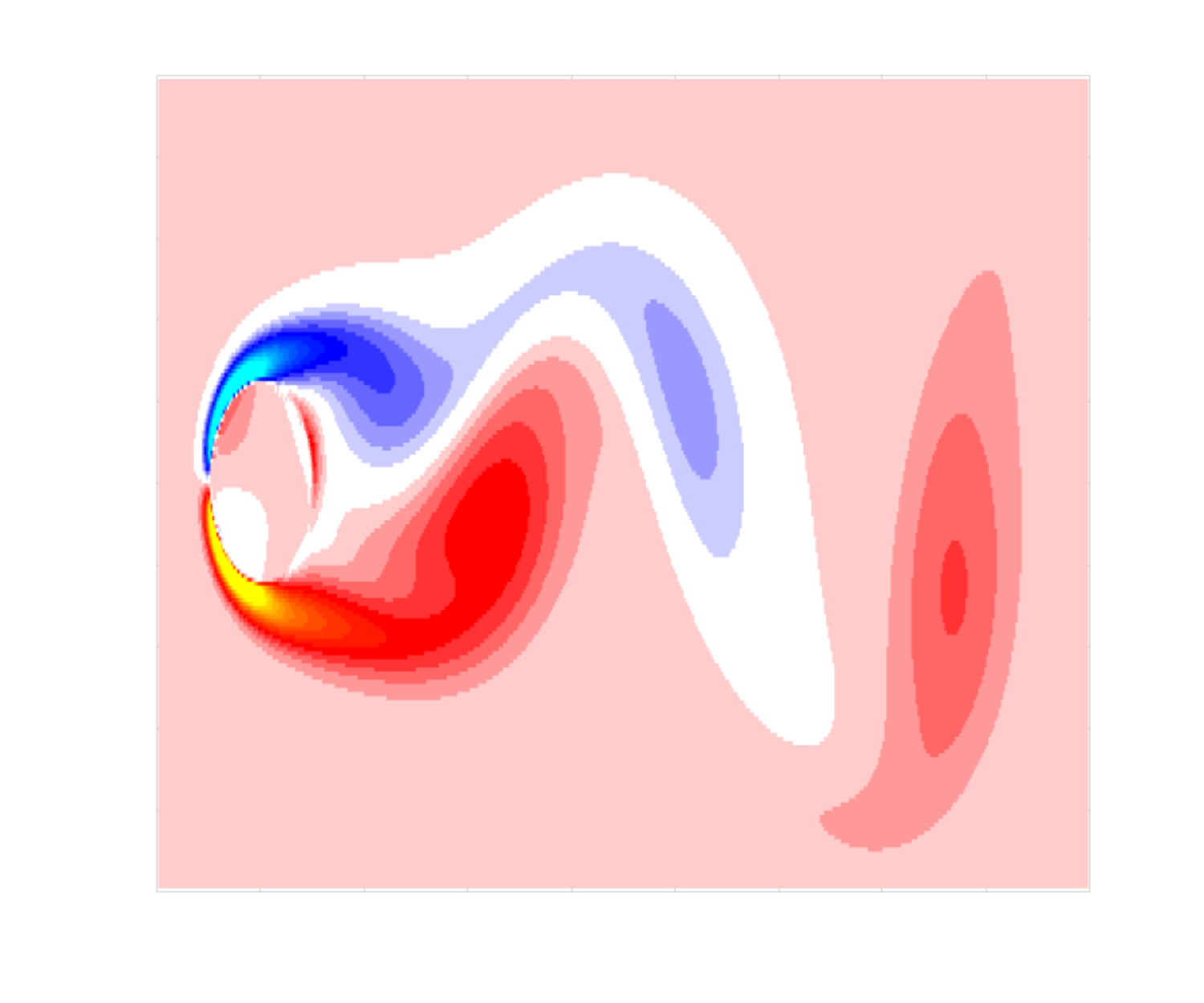}
    \caption{This figure presents prediction of the vorticity field at the same time points using two different kernel functions. The left column presents the prediction of, $x_{51}$. The middle column shows the reconstructions of the state, $x_{101}$, and the right column corresponds to the reconstruction of the state $x_{151}$.}
    \label{fig:prediction}
    \end{subfigure}
    \caption{This figure presents reconstruction and prediction of the vorticity of a fluid flow past a cylinder using two different kernel functions. The first rows contains the ground truth, the second rows leverages the Gaussian RBF kernel function, and the third row uses the exponential dot product kernel function.}
\end{figure}

\section{Numerical Example}\label{sec:example}

The website 
accompanying the textbook \citep{kutz2016dynamic} provides several data sets that serve as benchmarks for spectral decomposition approaches to nonlinear modeling, which have been released to the public through their website at \url{http://www.dmdbook.com/}. This section utilizes the cylinder flow data set to demonstrate the results of the developed DMD method. The cylinder flow example is numerically generated, and the data provided corresponds to a planar steady state flow of the system. The data set consists of $151$ snapshots, containing values of the vorticity of the flow at several mesh points in a plane, recorded every $0.02$ seconds. In this demonstration, snapshots $1$ through $30$ are used as the input data, and snapshots $2$ through $31$ are used as output data, assuming that the $i$th and $(i+1)$th snapshots satisfy $x_{i+1} = F(x_i)$ for some unknown nonlinear function $F$.

The Koopman modes, eigenvalues, and eigenfunctions are then computed using the developed technique and snapshots $1$ through $30$ are reconstructed using \eqref{eq:reconstruction}. DMD is implemented using the exponential dot product kernel, $K(x,y) = \exp(\frac{1}{\mu} x^T y)$ (with $\mu = 500$), and the Gaussian RBF kernel, $K(x,y) = \exp\left(-\frac{1}{\mu} \| x-y\|^2_2\right)$ (with $\mu = 500$). Each of the kernels result in  comparable reconstruction results that match with the input data as shown in Figure \ref{fig:reconstruction}. 

To further demonstrate the accuracy of the obtained finite-dimensional representation of the Koopman operator, snapshots $32$ through $151$ are \emph{predicted} using \eqref{eq:reconstruction}. As shown in Figure \ref{fig:prediction}, given the first $31$ snapshots, the developed DMD technique is able to accurately predict the remaining $120$ snapshots \emph{without the knowledge of the underlying physics, $F$.}

\section{Conclusion}\label{sec:conclusion}
This manuscript revisits theoretical assumptions concerning DMD of Koopman operators, including the existence of lattices of eigenfunctions, common eigenfunctions between Koopman operators, and boundedness and compactness of Koopman operators. Counterexamples that illustrate restrictiveness of the assumptions are provided for each of the assumptions. In particular, a major theoretical result is established to show that the native RKHS of the Gaussian RBF kernel function only supports bounded Koopman operators if the dynamics are affine. Moreover, a kernel-based DMD algorithm that simplifies the algorithm in \citep{williams2015data} and presents it in an operator theoretic context is developed and validated through simulations. 

\begin{ack}
This research was supported by the Air Force Office of Scientific Research (AFOSR) under contract numbers FA9550-20-1-0127 and FA9550-21-1-0134, and the National Science Foundation (NSF) under award numbers 2027976 and 1900364. Any opinions, findings and conclusions or recommendations expressed in this material are those of the author(s) and do not necessarily reflect the views of the sponsoring agencies.
\end{ack}

\small
\bibliography{references}

\end{document}